\documentclass[a4paper,12pt,reqno,openany]{amsart} 
\usepackage[left=2cm,right=2cm,top=2.5cm,bottom=2.5cm]{geometry}
\usepackage[english]{babel}
\usepackage[cal=dutchcal, scr=rsfso]{mathalpha}

\usepackage{color} 
\definecolor{Prune}{RGB}{99,0,60}
\usepackage{mdframed}
\usepackage{multirow} 
\usepackage{multicol} 
\usepackage{scrextend} 
\usepackage{tikz}
\usepackage[absolute]{textpos} 
\usepackage{colortbl}
\usepackage{array}
\usepackage{geometry}

\usepackage{etoolbox}
\usepackage{graphicx}
\usepackage{caption}
\usepackage{booktabs}
\definecolor{myprpl}{RGB}{255,0,130}
\usepackage[colorlinks, citecolor=cyan, linkcolor=myprpl]{hyperref}
\usepackage{caption}
\graphicspath{ {./images/} }
\usepackage{scrextend}
\usepackage{fancyhdr}
\usepackage{enumerate}

\usepackage{mathtools}
\usepackage{amsfonts}
\usepackage{amssymb}
\usepackage{amsmath}
\usepackage{amsthm}
\usepackage{float}
\usepackage{tikz-cd}

\DeclareFontFamily{U}{wncy}{}
\DeclareFontShape{U}{wncy}{m}{n}{<->wncyr10}{}
\DeclareSymbolFont{mcy}{U}{wncy}{m}{n}
\DeclareMathSymbol{\Sh}{\mathord}{mcy}{"58} 
\DeclareMathSymbol{\Be}{\mathord}{mcy}{"42} 

\newcommand\midotimes{\raisebox{0.25ex}{\scalebox{0.75}{$\bigotimes$}}}

\newcommand{\pro}{{ _{\textrm{pro}} }}
\newcommand{\ppro}{_{\textrm{pro}}}

\newtheoremstyle{myplain}
  {.5em}       
  {.5em}       
  {\itshape}   
  {0pt}        
  {\scshape}   
  {.}          
  {5pt plus 1pt minus 1pt} 
  {}           

\newtheoremstyle{mydefinition}
  {.5em}       
  {.5em}       
  {\normalfont}
  {0pt}        
  {\bfseries}  
  {.}          
  {5pt plus 1pt minus 1pt} 
  {}           

\newtheoremstyle{myremark}
  {.5em}       
  {.5em}       
  {\normalfont}
  {0pt}        
  {\itshape}   
  {.}          
  {5pt plus 1pt minus 1pt} 
  {}           

\numberwithin{equation}{section}

\theoremstyle{myplain}
\newtheorem{thm}{Theorem}[section]
\newtheorem{lem}[thm]{Lemma}
\newtheorem{cor}[thm]{Corollary}
\newtheorem{prop}[thm]{Proposition}

\theoremstyle{mydefinition}
\newtheorem{defn}[thm]{Definition}
\newtheorem{exmp}[thm]{Example}

\theoremstyle{myremark}
\newtheorem{rem}[thm]{Remark}

\newtheorem*{intprf}{Proof}
\newtheorem*{intprfsk}{Sketch of Proof}

\newenvironment{prf}[1][]
    {\ifstrempty{#1}{\begin{intprf}}
                    {\begin{intprf}[\textit{#1}]}}
    {\qed\end{intprf}}

\title[Brauer-Manin obstructions for commutative affine groups]
{{\large{\centerline{Brauer-Manin obstructions for}\centerline{homogeneous spaces of commutative}\centerline{affine algebraic groups over global fields}}}
\hspace{0pt}
\\\;
\\{$\;$\normalfont{Azur ĐonlagiĆ}$\,^\star$}}
\thanks{$^\star\,$Part of the author's PhD project at Universit\'e Paris-Saclay under the guidance of prof. David Harari and funded through the PhD Track of the Fondation Math\'ematique Jacques Hadamard.}

\begin{document}

\begin{abstract}
Questions related to Brauer-Manin obstructions to the Hasse principle and weak approximation for homogeneous spaces of tori over a number field are well-studied, generally using arithmetic duality theorems, starting with works of Sansuc and of Colliot-Th\'el\`ene. In this article, we prove the analogous statements (and include obstructions to strong approximation) in the general case of a commutative affine group scheme $G$ of finite type over a global field in any characteristic. We also study finiteness of the $\Sh^2_S(\widehat{G})$ and $\Sh^2_\omega(\widehat{G})$ kernels of the Cartier dual of $G$. All this is made possible by some recent theoretical advancements in positive~characteristic, namely the finiteness theorems of B. Conrad and the generalized Tate duality of Z. Rosengarten. 
\end{abstract}
\maketitle

\vspace{-10pt}
\tableofcontents
\vspace{-25pt}

\section{Introduction and notation}

In this article, an ``algebraic group'' means a group scheme of finite type over a field. Fix a global field $k$ (of any characteristic, although we will mostly concern ourselves with positive characteristic) and a commutative affine algebraic group $G$ over it. Let $\Omega$ be the set of places of $k$. For any subset $S\subseteq\Omega$, we write $k_S\coloneqq \prod_{v\in S} k_v$ and $k^S\coloneqq k_{\Omega\setminus S}$, and similarly $\mathbf A = \mathbf A_S\times\mathbf A^S$ for the adeles of $k$ with their respective topology. In particular, $\mathbf A_S = k_S$ for finite $S$. Whenever we consider $\mathbf A_S$ (or $\mathbf A^S$) as a $k$-algebra, we do it via the composition $k\rightarrow\mathbf A\rightarrow\mathbf A_S$.

Let $X$ be a scheme of finite type over $k$. For any topological $k$-algebra $R$, we endow the $R$-points $X(R)$ with their natural topology as explained in \cite{ConAd}. We denote by $X(k)$ both the $k$-points of $X$ and their natural image in $X(R)$, implicit in the context. Recall that $X$ satisfies the \textit{Hasse principle} if the logical implication $X(k_\Omega)\neq\varnothing\Longrightarrow X(k)\neq\varnothing$ holds true (its converse clearly always holds). If the stronger property $\overline{X(k)} = X(k_\Omega)$ holds true, where the bar denotes closure with respect to the product topology on $k_\Omega$, we say that $X$ satisfies the \textit{weak approximation property}. A closely related property is the \textit{strong approximation property with respect to $S\subset\Omega$}, which is satisfied by $X$ if $\overline{X(k)} = X(\mathbf A_S)$ with respect to the adelic topology. Strong approximation can usually be proven at most for proper cofinite subsets $S$ of $\Omega$. Note that weak approximation is essentially equivalent to strong approximation holding true for all finite $S$, while the name ``strong'' is generally reserved for cofinite $S$, which explains their naming convention.

Most families of varieties do not satisfy the Hasse principle nor weak approximation, however this failure can sometimes (notably in low-dimensional cases or for some homogeneous spaces of algebraic groups) be explained by the existence of a so-called Brauer-Manin obstruction, to be introduced formally in the next section. To be precise, given a collection of varieties, we say that the Brauer-Manin obstruction is the \textit{only} obstruction to the Hasse principle (resp. weak/strong approximation) for varieties in this collection if every variety in this collection for which the Brauer-Manin obstruction vanishes satisfies the Hasse principle (resp. weak/strong approximation, in a slightly more subtle sense). In his 1981 paper \cite{San81} (\textsection 8), Sansuc showed this property for principal homogeneous spaces of connected affine algebraic groups (without a factor of type $E_8$) over a number field $k$. The main idea of his proof is to, crucially using the connectedness assumption, reduce to the essential case of a torus $T$, then to use a comparison between the Brauer-Manin pairing (which defines the obstruction) and the Poitou-Tate pairing for the finite kernel of an isogeny of tori $\widetilde{T}\rightarrow T$ with $\widetilde{T}$ quasi-trivial. The Poitou-Tate pairing is known to be nondegenerate, which then implies the desired result.

Alternatively (as explained in \cite{Sko01}, \textsection 6.2), it is possible to apply a more general version of the Poitou-Tate pairing directly to $T$, then show compatibility with the Brauer-Manin pairing which immediately gives the result for tori. We generalize this approach using the Tate pairings recently introduced by Rosengarten in \cite{RosTD} for any commutative affine algebraic group $G$, not necessarily smooth, over a global field in any characteristic (see \cite{RosTD}, Appendix G; all cited results of this paper remain true over number fields, up to everywhere replacing $G(\mathbf R)$ and $G(\mathbf C)$ by the Tate cohomology group $\mathrm{\widehat H}^0(\mathbf R,G)$ and $0$, respectively), as follows:
\vfill

We explain all the discussed Brauer-Manin obstructions for a principal homogeneous space $X$ of $G$ in Section 2 and immediately construct the functorial comparison map which allows us to reinterpret them in terms of the Cartier dual $\widehat{G}$ of $G$ and the cup product in cohomology. This gives, in Section 3, a characteristic-free proof (Theorem \ref{thmmainglobal}) for homogeneous spaces $X$ of all commutative affine algebraic groups $G$ of the uniqueness of the Brauer-Manin obstruction to the Hasse principle. The present difficulties are in constructing~the~comparison map and then checking the compatibility of the two pairings. All of the obstacles come~from~the use of flat cohomology instead of \'etale cohomology. For example, the comparison map must be constructed without appeal to Rosenlicht's lemma, which is not available outside of \'etale topology.

In Section 4, we show uniqueness on $X$ of the Brauer-Manin obstruction to weak and strong approximation. Unlike Sansuc, we do not assume connectedness. However,~as~remarked above, the Poitou-Tate theory cannot recover information about $\mathbf R$- and $\mathbf C$-points; thus our statements in characteristic $0$ are limited to non-Archimedean places. For a sheaf $\mathcal F$ on the fppf site of $k$,
we define the Tate-Shafarevich kernels $\Sh^i$ and $S$-kernels $\Sh^i_S$ for any set $S\subseteq\Omega$ by letting
\vspace{-11pt}

$$\Sh^i_S(\mathcal F)\coloneqq \ker\left(\mathrm{H}^i(k,\mathcal F)\longrightarrow\prod_{v\in\Omega\setminus S}\mathrm{H}^i(k_v,\mathcal F)\right)$$
\vspace{-7pt}

\noindent
and $\Sh^i(\mathcal F)\coloneqq\Sh^i_\varnothing(\mathcal F)$, all maps being induced by the completions $k\rightarrow k_v$. We~consider~also the $\omega$-kernel $\Sh_\omega^i(\mathcal F)\coloneqq\bigcup_{\textrm{finite }S}\Sh_S^i(\mathcal F)$ and similarly introduce $\Sh_f^i(\mathcal F)\coloneqq\bigcup_{\textrm{finite }S\subseteq \Omega_f}\Sh_S^i(\mathcal F)$, where $\Omega_f = \Omega\setminus\Omega_\infty$ denotes the finite places of $k$ (so $\Sh_f^i = \Sh_\omega^i$ if $\mathrm{char}\,k > 0$). The main result (Theorem \ref{thmmainlocal}) of Section 4 is based on the following two exact sequences, where $\Omega\neq S\subseteq\Omega_f$,\\
\vspace{-21pt}

$$0\longrightarrow \overline{G(k)}^{_S}\longrightarrow G(\mathbf A_S)\longrightarrow \Sh_S^2(\widehat{G})^*\longrightarrow \Sh^2(\widehat{G})^*\longrightarrow 0$$
$$0\longrightarrow \overline{G(k)}\longrightarrow \prod\nolimits_{v\in\Omega_f} G(k_v)\longrightarrow \Sh_f^2(\widehat{G})^*\longrightarrow  \Sh^2(\widehat{G})^*\longrightarrow 0$$
\vspace{-11pt}

\noindent
in which the groups on the left are closures in ($S$-)adelic and product~topologies, respectively, and $^*$ denotes the algebraic dual. Since the left halves of the two sequences suffice~to~deduce~the Brauer-Manin statements, the proof of exactness of the right half of the second sequence is done only in Section 5. It is shown to be equivalent to compactness of the quotient $\prod\nolimits_{v\in\Omega_f} G(k_v)/\overline{G(k)}$. This quotient is also shown (Theorem \ref{thmsha2fin}) to be finite when the group $G$ has no non-finite wound unipotent quotient; otherwise a counterexample can be proven to exist by the work of Oesterl\'e. We study such counterexamples explicitly in Section 6.

All cohomology considered in this article is fppf cohomology. The algebraic (resp. separable) closure of $k$ is denoted by $\overline{k}$ (resp. $k^s$). We denote the additive (resp. multiplicative) algebraic group by $\mathbf G_\mathrm{a}$ (resp. $\mathbf G_\mathrm{m}$) and we identify them with the fppf sheaves they represent.~We~do~not explicitly write base changes (e.g. $\mathbf G_{\mathrm{m},k}$ instead of $\mathbf G_\mathrm{m}$) if they are clear in context. 
We heavily borrow from Rosengarten's remarkable paper \cite{RosTD} and so adopt the same notation wherever possible. This includes the following notation:
\begin{itemize}
    \item for Cartier duals of fppf $k$-sheaves, $\widehat{\mathcal F} = \mathcal{Hom}_k(\mathcal F, \mathbf G_\mathrm{m})$
    \item for algebraic duals of Abelian groups, $A^* = \mathrm{Hom}(A, \mathbf Q/\mathbf Z)$, an exact functor
    \item for topological duals of topological groups, $A^D = \mathrm{Hom}_{cont}(A, \mathbf R/\mathbf Z)$
    \item for profinite completions of (discrete or topological) groups, $A\pro\coloneqq\varprojlim A/U$ taken over open subgroups $U\subseteq A$ of finite index
\end{itemize}
Note that for all topological groups which we will consider (discrete torsion, or profinite,~or~those of finite exponent) we have $A^D = \mathrm{Hom}_{cont}(A, (\mathbf Q/\mathbf Z)_{disc}) = \mathrm{Hom}_{loc.const}(A, \mathbf Q/\mathbf Z)$.


\section{The Brauer-Manin obstruction and comparison maps}

Let $k$ be a global field. Given any completion $k_v$ of $k$, we write $\mathrm{inv}_v : \mathrm{Br}(k_v)\rightarrow\mathbf Q/\mathbf Z$ for the invariant map of class field theory (an isomorphism for finite places $v$, by a theorem of Hasse). Recall the following sequence, which is exact by the Brauer-Hasse-Noether theorem:
\vspace{-10pt}

$$0\rightarrow \mathrm{Br}(k)\rightarrow\bigoplus\nolimits_v\mathrm{Br}(k_v)\xrightarrow{\;\sum_v\mathrm{inv}_v\;}\mathbf Q/\mathbf Z\rightarrow 0$$

\begin{defn}
Consider a scheme $X$ of finite type over $k$ and its (cohomological)~Brauer~group $\mathrm{Br}(X) = \mathrm{H}^2(X,\mathbf G_{\mathrm{m}})$. The \textit{Brauer-Manin pairing} $\langle-,-\rangle_{BM}$ related to $X$ is the composition
$$X(\mathbf A)\times\mathrm{Br}(X)\longrightarrow \bigoplus\nolimits_v\mathrm{Br}(k_v)\rightarrow\mathbf Q/\mathbf Z\, ,\;\;\;\; \langle(P_v),A\rangle_{BM} = \sum\nolimits_v \mathrm{inv}_v(A(P_v))$$
where $X(\mathbf A)\subseteq X(k_\Omega) = \prod_v\,\!\! X(k_v)$ and the symbol $A(P_v)$ denotes the pullback of the local~image $A_v\in\mathrm{Br}(X_{k_v})$ to $\mathrm{Br}(k_v)$ by $P_v$. The pairing is well-defined because the sum on the right is finite by Proposition 8.2.1 in \cite{Poo17}. It is linear in the right element.
\end{defn}

If $(P_v)$ is determined by some $P\in X(k)$, then $A(P_v)$ is an image of $A(P)\in\mathrm{Br}(k)$ and the pair $\langle P,A\rangle_{BM}$ vanishes by the Brauer-Hasse-Noether theorem. That is, $X(k)$ lies in the left kernel $X(\mathbf A)^{\mathrm{Br}}$ of $\langle -,-\rangle_{BM}$ (the set of points ``orthogonal'' to $\mathrm{Br}(X)$ with respect to this pairing).
Therefore, if $X(\mathbf A)^{\mathrm{Br}}$ is empty, then $X(k)$ must be empty as well. This general condition is called the \textit{Brauer-Manin obstruction to the Hasse principle} on $X$, as it precludes the existence of rational points even if $X(\mathbf A)$ is nonempty.
Conversely, we would like to know whether, if there exists an adelic point $(P_v)\in X(\mathbf A)$ orthogonal to (some particular part $\mathcal B(X)$ of) $\mathrm{Br}(X)$, we can conclude the non-emptiness of $X(k)$. If this property ($X(\mathbf A)^{\mathcal B}\neq\varnothing\Longrightarrow X(k)\neq\varnothing$) is true for all varieties in some family, then we say that the Brauer-Manin obstruction (with respect to $\mathcal B$) is the \textit{only obstruction to the Hasse principle} for this family.

Finally, the Brauer-Manin pairing is continuous with respect to the adelic topology (\cite{Poo17}, Corollary 8.2.11), hence the left kernel $X(\mathbf A)^{\mathrm{Br}}$ of $\langle -,-\rangle_{BM}$ is closed and contains $\overline{X(k)}$. Therefore, the condition $X(\mathbf A)^{\mathrm{Br}}\neq X(\mathbf A)$, if it holds, is called the \textit{Brauer-Manin obstruction to strong approximation} of adelic points on $X$. More generally (and much more realistically), we consider only some natural subgroup $\mathcal B_S$ of $\mathrm{Br}$ such that the pairing $\langle -,-\rangle_{BM}$ descends to a pairing $X(\mathbf A_S)\times\mathcal B_S$ for a subset $S\subseteq\Omega$. Then we may hope that, for a family of varieties, the Brauer-Manin obstruction with respect to $\mathcal B_S$ is the only obstruction to strong approximation of points in $X(\mathbf A_S)$. Of course, the analogous statement for weak approximation is equivalent to this statement taken over all finite subsets $S$ of $\Omega$.

\begin{rem}\label{remgeomint}
As in the Hasse principle, we would often like to relax the condition $X(\mathbf A)\neq\varnothing$ to $X(k_{\Omega})\neq\varnothing$. This is immediate if the finite-type $k$-scheme $X$ is proper, as then $X(\mathbf A) = X(k_{\Omega})$ ($X(\mathcal O_v) = X(k_v)$ for almost all $v$ by ``spreading out'' and the valuative criterion of properness), and also when $X$ is geometrically integral (because then $X(k_\Omega)\neq\varnothing$ implies $X(\mathbf A)\neq\varnothing$ by the proof of Theorem 7.7.2 in \cite{Poo17}).
Otherwise, when $k$ is a number field, we may~pass~to~a smooth compactification $X_{\! c}$ of $X$ and consider the ``unramified Brauer group'' $\mathrm{Br}(X)_{\mathrm{nr}} = \mathrm{Br}(X_{\! c})$ so that a pairing $X(k_{\Omega})\times\mathrm{Br}(X)_{\mathrm{nr}}\rightarrow\mathbf Q/\mathbf Z$ is well-defined. It is then possible to speak directly of an ``obstruction to weak approximation'' given by this pairing (see \cite{Sko01}, \textsection 5.2).

Smooth compactifications are in general not available over fields of positive characteristic. However, as explained below, the role of $\mathrm{Br}(X)_{\mathrm{nr}}$ can be played by the groups $\Be_\omega(X)$~and~$\Be_f(X)$ in a weaker sense. This will be sufficient to discuss weak approximation in our setting. 
\end{rem}

\begin{defn}
We now consider variants of the Brauer group which are especially convenient in the context of homogeneous spaces of commutative affine algebraic groups. Define
\vspace{-13pt}

$$\Be_S(X)\coloneqq\mathrm{ker}\left(\frac{\mathrm{Br}(X)}{\mathrm{im}\,\mathrm{Br}(k)}\longrightarrow\prod_{v\in\Omega\setminus S}\frac{\mathrm{Br}(X_{k_v})}{\mathrm{im}\,\mathrm{Br}(k_v)}\right)$$ 
\vspace{-10pt}

\noindent
and also $\Be(X)$, $\Be_\omega(X)$ and $\Be_f(X)$ in analogy with the Tate-Shafarevich kernels of Section 1.
\end{defn}

Suppose that $X(\mathbf A)\neq\varnothing$. Then $X(\mathbf A^S)\neq\varnothing$ and the map $X(\mathbf A)\rightarrow X(\mathbf A_S)$ is a surjection for every $S\subseteq\Omega$. Also, $X(k_v)\neq\varnothing$ so that the maps $\mathrm{Br}(k_v)\rightarrow\mathrm{Br}(X_{k_v})$ are injections (any element of $X(k_v)$ gives a left inverse). If $A$ represents an element of $\Be(X)$, then $A_v$ lies by definition in the image of this map for every $v$ and $A(P_v)$ is thus independent of the choice of $P_v$.

It follows that the value of $\langle (P_v),A\rangle_{BM}$ is independent of the choice of $(P_v)\in X(\mathbf{A})$. However, at least one such choice exists by assumption, and since $\mathrm{Br}(k)$ belongs to the right kernel of $\langle (P_v),A\rangle_{BM}$ by the Brauer-Hasse-Noether theorem, this implies the existence of a well-defined map $\Be(X)\rightarrow\mathbf Q/\mathbf Z$. If this map is nontrivial, then $X(k) = \varnothing$ (otherwise we may use any $P\in X(k)$ in its construction, and we end up with a zero map). We may therefore call it the \textit{Brauer-Manin obstruction to the Hasse principle given by $\Be(X)$}.

More generally, if $A$ represents an element of $\Be_S(X)$ for some $S\in\Omega$, the value of $\langle (P_v),A\rangle_{BM}$ depends only on ($A$ and) the components $P_v$ for $v\in S$. This can be written as a diagram 
\vspace{-4pt}

\begin{center}\begin{tikzcd}[column sep=tiny]
    X(\mathbf A)\arrow[d, two heads]
\hspace{-5pt}&\hspace{-10pt}{\times}\hspace{-10pt}&\hspace{-5pt}
    \dfrac{\mathrm{Br}(X)}{\mathrm{im}\,\mathrm{Br}(k)}
    \arrow[rrrrrrrrrrrr, "{\langle -,-\rangle_{BM}}"]
    &&&&&&&&&&&& \mathbf Q/\mathbf Z\arrow[d, equal]\\
    X(\mathbf A_S)
\hspace{-5pt}&\hspace{-10pt}{\times}\hspace{-10pt}&\hspace{-5pt}
    \Be_S(X)\arrow[u, hook]
    \arrow[rrrrrrrrrrrr]
    &&&&&&&&&&&& \mathbf Q/\mathbf Z
\end{tikzcd}\end{center}
\vspace{-2pt}

\noindent
and the induced function $X(\mathbf A_S)\rightarrow\Be_S(X)^*$ is nonzero only if strong approximation for $S$ fails. It is the \textit{Brauer-Manin obstruction to strong approximation with respect to $S$ given by $\Be_S(X)$}. The statement for $\Be$ can be seen as the case $S = \varnothing$ (where $X(\textbf{0}) = \ast$ for the zero ring $\textbf{0}$).\;\;\;\;\;\;

Taking the inverse limit of $X(\mathbf A_S)\rightarrow\Be_S(X)^*$ over finite $S$, we arrive at $X(k_\Omega)\rightarrow\Be_\omega(X)^*$, the \textit{Brauer-Manin obstruction to weak approximation given by $\Be_\omega(X)$}. Note that, although~this pairing is defined for all points in $X(k_\Omega)$, we do not claim that it exists unless $X(\mathbf A)\neq\varnothing$. By Remark \ref{remgeomint}, this note is important only when $X$ is not geometrically integral: In our paper,~this situation occurs when $X$ is a torsor of a nonsmooth algebraic group.

We may similarly consider the map $X(k_{\Omega_f})\rightarrow\Be_f(X)^*$ in characteristic $0$ (otherwise $\Omega_f = \Omega$ and $\Be_f = \Be_\omega$).
In the following two sections, we prove that the obstructions related to $\Be$, $\Be_f$ and $\Be_S$ (when $\Omega\neq S\subseteq\Omega_f$) are the only ones for homogeneous spaces of commutative affine algebraic groups over $k$. These are Theorems \ref{thmmainglobal} and \ref{thmmainlocal}.
The main idea of these proofs is to relate the Brauer-Manin pairing to some form of the Poitou-Tate pairing via a comparison~map $\Sh^2(\widehat{G})\rightarrow\Be(X)$, where $X$ is a torsor of $G$, which we construct in the remainder of this section.

\begin{rem}
The notation for $\Be$ and $\Be_S$ is not completely consistent. Some authors (notably~in \cite{San81}, \cite{Brv96}, \cite{Sko01}) working over number fields $k$ often replace $\mathrm{Br}(X)$ in the definitions of $\Be$ and $\Be_S$ by its ``algebraic part'' $\mathrm{Br}_1(X)\coloneqq\mathrm{ker}(\mathrm{Br}(X)\rightarrow\mathrm{Br}(X_{k^s}))$. This is convenient, as the term $\mathrm{Br}_1(X)$ appears in the long exact sequence coming from the Hochschild-Serre spectral sequence. Then a variant of Rosenlicht's lemma in Galois cohomology is used to produce the comparison map mentioned above. In particular, the resulting map is then an isomorphism.

We cannot use Rosenlicht's lemma, as working with nonsmooth varieties and flat cohomology prevents us from reducing to Galois cohomology, therefore we do not bother with distinguishing the algebraic part of the Brauer group. However, we remark here that the two definitions of $\Be$ (and $\Be_S$, as soon as $S\neq\Omega$) do in fact agree: Indeed, given an element $A\in\mathrm{Br}(X)$ representing a class in $\Be_S(X)$, then $A_{k_v}\in\mathrm{Br}(k_v)$ for any fixed $v\notin S$ and thus $A_K = 0$ for a finite separable extension $K/k_v$. Because the extension $K/k$ is separable, a limiting argument gives a smooth finite-type $k$-algebra $R$ with $A_R = 0$. As $\mathrm{Spec}(R)$ admits a $k^s$-point, we conclude that $A_{k^s} = 0$.
\end{rem}

To construct the desired comparison map, we need to introduce an intermediate object:

\begin{defn}\label{defxprim}
Let $X$ be a torsor of a commutative affine algebraic group $G$ over $k$. Then by definition, there is an isomorphism $G\times X\xrightarrow{\;\sim\;} X\times X$, which gives a map:
$$q : X\times X\xrightarrow{\;\sim\;} G\times X\xrightarrow{\mathrm{pr}_G} G,\;\;\;\; q(x_1,x_2) = x_1-x_2$$
Let $R$ be a $k$-algebra. We write $X'(R)$ for the set of all $\mathrm{Sch}/R$ morphisms $f : X_R\rightarrow\mathbf G_{\mathrm{m},R}$ such that there exists a $\mathrm{Sch}/R$ morphism $\tilde{f} : G_R\rightarrow\mathbf G_{\mathrm{m},R}$ which makes the diagram
\begin{center}\begin{tikzcd}
    X_R\times X_R\arrow[r, "f\times f"]\arrow[d, "q"] & \mathbf{G}_{\mathrm{m},R}\times \mathbf{G}_{\mathrm{m},R}\arrow[d, "\mathrm{mul}\,\circ\,(\mathrm{id}\times\mathrm{inv})"]\\
    G_R\arrow[r,  "\tilde{f}"] & \mathbf{G}_{\mathrm{m},R}
\end{tikzcd}\end{center}
commute. If such an $\tilde{f}$ exists, it's uniquely determined by $f$ because $q$ is faithfully flat. The identification of $f$ with $(f,\tilde{f})$ interprets $X'(R)$ as the pullback:
\begin{center}\begin{tikzcd}[column sep = huge]
    X'(R)\arrow[rr, "f\,\mapsto\,\tilde{f}"]\arrow[d, hook] & & \mathrm{Mor}_R(G_R, \mathbf{G}_{\mathrm{m},R})\arrow[d, "f\,\mapsto\,f\,\circ\,q"]\\
    \mathrm{Mor}_R(X_R, \mathbf{G}_{\mathrm{m},R})\arrow[rr,  "f\,\mapsto\,\mathrm{mul}\,\circ\,(f\,\times\,(\mathrm{inv}\,\circ\,f))"] & & \mathrm{Mor}_R(X_R\times X_R, \mathbf{G}_{\mathrm{m},R})
\end{tikzcd}\end{center}
This shows that $X' \hspace{-0.7pt}=\hspace{-0.7pt} \mathrm{ker}(\mathcal{Mor}(G, \mathbf{G}_{\mathrm{m}}){\times}\mathcal{Mor}(X, \mathbf{G}_{\mathrm{m}})\rightrightarrows\mathcal{Mor}(X{\times} X, \mathbf{G}_{\mathrm{m}}))$ is an fppf sheaf~on~$k$.
Here $\mathrm{Mor}$ (resp. $\mathcal{Mor}$) denotes the group (resp. sheaf) of scheme morphisms to a group scheme.
\end{defn}

If $X = G$, then the defining condition $f(g_1)-f(g_2) = \tilde{f}(g_1-g_2)$ immediately implies that $f(g) =  f(0)+\tilde{f}(g)$ and that $\tilde{f}$ is a homomorphism. Now, for an arbitrary torsor $X$, it's easy to see that $(f,\tilde{f})\in X'(R)$ gives $(\tilde{f},\tilde{f})\in G'(R)$ and hence $\tilde{f}$ must again be a homomorphism. This implies that $X'\rightarrow\mathcal{Mor}(G, \mathbf{G}_{\mathrm{m}})$ factors through $\widehat{G}$. Moreover, there is a canonical inclusion of $\mathbf{G}_{\mathrm{m}}$ in $X'$ as constant maps. By the observation above, the sequence
\vspace{-10pt}

\begin{equation}\label{eqxprim}
0\rightarrow \mathbf{G}_{\mathrm{m}}\rightarrow X'\rightarrow \widehat{G}\rightarrow 0
\end{equation}
\vspace{-15pt}

\noindent
is exact when $X = G$. However, since $X_{k'}$ becomes isomorphic to $G_{k'}$ in $G\mathrm{-Sch}/k'$ over some extension $k'/k$, the sequence must be exact for all $X$. Next, $k$ is global, so (by \cite{CF67}, Chapter VII, \textsection 11.4) we have $\mathrm{H}^3(k, \mathbf{G}_{\mathrm{m}}) = 0$ and the sequence $\mathrm{H}^2(k, \mathbf{G}_{\mathrm{m}})\rightarrow\mathrm{H}^2(k, X')\rightarrow\mathrm{H}^2(k, \widehat{G})\rightarrow 0$ is exact. Finally, this gives the following canonical compositions (where the second map comes from the long exact sequence of the Leray spectral sequence associated to $X\rightarrow k$ and $\mathbf{G}_{\mathrm{m}}$):
\vspace{-10pt}

$$\mathrm{H}^2(k, X')\longrightarrow\mathrm{H}^2(k, \mathcal{Mor}(X, \mathbf{G}_{\mathrm{m}}))\longrightarrow\mathrm{Br}(X)
\;\textrm{ and so }\;
\mathrm{H}^2(k, \widehat{G}) = \frac{\mathrm{H}^2(k, X')}{\mathrm{im}\,\mathrm{H}^2(k, \mathbf{G}_{\mathrm{m}})}\longrightarrow\frac{\mathrm{Br}(X)}{\mathrm{im}\,\mathrm{Br}(k)}$$

\begin{rem}
Here, the application of Rosenlicht's lemma in the number field case (which can be interpreted as saying that $0\rightarrow\mathbf{G}_{\mathrm{m}}\rightarrow\mathcal{Mor}(X, \mathbf{G}_{\mathrm{m}}){\times}\mathcal{Mor}(Y, \mathbf{G}_{\mathrm{m}})\rightarrow\mathcal{Mor}(X{\times} Y, \mathbf{G}_{\mathrm{m}})\rightarrow 0$ is a short exact sequence of \'etale presheaves over $k$ for $X,Y$ of finite type over $k$) corresponds~to~an isomorphism $\mathrm{H}^2(k, \mathcal{Mor}(X, \mathbf{G}_{\mathrm{m}}))\xrightarrow{\sim} \mathrm{H}^2(k, \widehat{G})$ for \'etale sheaves. Its inverse is then composed with the second map, coming from the Hochschild-Serre spectral sequence (see \cite{Sko01}, \textsection 6.2). 
\end{rem}

Returing to our setting: From the final map constructed above, we deduce, by functoriality~for $k$ and all $k_v$, the promised homomorphisms $\Phi_X : \Sh^2(\widehat{G})\rightarrow\Be(X)$ and $\Phi_{X,S} : \Sh^2_S(\widehat{G})\rightarrow\Be_S(X)$. It remains to connect them with the Brauer-Manin obstruction via the following statements:

\begin{lem}\label{lembmglob}
Let $X$ be a torsor of $G$ such that $X(\mathbf A)\neq\varnothing$. The exact sequence \eqref{eqxprim} induces~a natural map $\Sh^2(\widehat{G})\rightarrow\mathbf Q/\mathbf Z$ which agrees with the composition $\Sh^2(\widehat{G})\rightarrow\Be(X)\rightarrow\mathbf Q/\mathbf Z$~of~$\Phi_X$ with the Brauer-Manin obstruction map.
\end{lem}
\begin{prf}
Consider the commutative diagram
\vspace{-33 pt}

\begin{center}\begin{tikzcd}[row sep = tiny, column sep = tiny]
      & &  & & \Be(X)\arrow[ddd, hook]\\
    & & &  & & \Sh^2(\widehat{G})\arrow[ddd, hook]\arrow[ul, "\Phi_X", swap]\\ \\
      \mathrm{Br}(k)\arrow[rr, hook]\arrow[ddd, hook] & & \mathrm{Br}(X)\arrow[rr, two heads]\arrow[ddd] & & \dfrac{\mathrm{Br}(X)}{\mathrm{Br}(k)}\arrow[ddd]\\
    & \mathrm{H}^2(k,\mathbf G_{\mathrm{m}})\arrow[rr, hook]\arrow[ddd, hook]\arrow[ul, equal] & & \mathrm{H}^2(k,X')\arrow[rr, two heads]\arrow[ddd]\arrow[ul] & & \mathrm{H}^2(k,\widehat{G})\arrow[ddd]\arrow[ul]\\ \\
      \prod_v\mathrm{Br}(k_v)\arrow[rr, hook]\arrow[ddd] & & \prod_v\mathrm{Br}(X_{k_v})\arrow[rr, two heads] & & \prod_v\dfrac{\mathrm{Br}(X_{k_v})}{\mathrm{Br}(k_v)}\\
    & \prod_v\mathrm{H}^2(k_v,\mathbf G_{\mathrm{m}})
    \arrow[rr, hook]\arrow[ddd]\arrow[ul, equal] & & \prod_v\mathrm{H}^2(k_v,X')\arrow[rr, two heads]\arrow[ul] & & \prod_v\mathrm{H}^2(k_v,\widehat{G})\arrow[ul]\\ \\
      \dfrac{\prod_v\mathrm{Br}(k_v)}{\mathrm{Br}(k)}\\
    & \dfrac{\prod_v\mathrm{H}^2(k_v,\mathbf G_{\mathrm{m}})}{\mathrm{H}^2(k,\mathbf G_{\mathrm{m}})}\arrow[ul, equal]
\end{tikzcd}\end{center}
\vspace{-50pt}

\begin{equation}\label{diagbmobstr}
\end{equation}
\vspace{10pt}

\noindent
with diagonal arrows induced by the construction of $\Phi_X$ and, for the lower right diagonal arrow, by $\mathrm{H}^3(k_v,\mathbf G_{\mathrm{m}}) = 0$ (local duality). It has exact rows by exactness of products. All maps drawn as injections or surjections are either already known to be so (by the Brauer-Hasse-Noether theorem or by the condition $X(\mathbf A)\neq\varnothing$), or their injectivity immediately follows from that of other maps in the diagram. Applying the snake lemma twice above, we get two connecting homomorphisms $\delta_{BM}:\Be(X)\!\rightarrow\!\prod_v\mathrm{Br}(k_v)/\mathrm{Br}(k)$ and $\delta:\Sh^2(\widehat{G})\!\rightarrow\!\prod_v\mathrm{H}^2(k_v,\mathbf G_{\mathrm{m}})/\mathrm{H}^2(k,\mathbf G_{\mathrm{m}})$. But $\delta_{BM}$ agrees with the construction of the Brauer-Manin obstruction related to $\Be(X)$, hence lands inside $\bigoplus_v\mathrm{Br}(k_v)/\mathrm{Br}(k)$.
By functoriality of the snake lemma, the following commutes:
\begin{center}\begin{tikzcd}[column sep = huge]
    \Sh^2(\widehat{G})\arrow[rd, dashed]\arrow[r, "\delta"]\arrow[d, "\Phi_X"] & \dfrac{\prod_v\mathrm{H}^2(k_v,\mathbf G_{\mathrm{m}})}{\mathrm{H}^2(k,\mathbf G_{\mathrm{m}})}\arrow[r, equal] & \dfrac{\prod_v\mathrm{Br}(k_v)}{\mathrm{Br}(k)}\\
    \Be(X)\arrow[r, "\textrm{BM obstruction}"] & \dfrac{\bigoplus_v\mathrm{Br}(k_v)}{\mathrm{Br}(k)}\arrow[ur, hook]\arrow[r, hook, "\sum_v\mathrm{inv}_v"] & \mathbf Q/\mathbf Z
\end{tikzcd}\end{center}
Clearly, $\delta$ lifts to the dashed arrow, which shows that it lands in the quotient of the direct sum. Its composition with the sum of the invariant maps gives the desired map.
\end{prf}


This lemma will be crucial in the next section and the statements about the Hasse principle. Meanwhile, we show another lemma to be used later for statements about strong approximation: Note that we have defined the obstruction to strong approximation without requiring $X(k)\neq\varnothing$, just to highlight that this condition is logically implied by the denseness of $X(k)$ in $X(\mathbf A_S)\neq\varnothing$. However, we lose no generality in assuming this and we will in fact need this condition to apply the theory of local duality. Therefore, we now assume $X(k)\neq\varnothing$ holds for the remainder of this section, where $X$ is a torsor of $G$. Equivalently, $X\simeq G$ over $k$, so that we may replace $X$ by $G$ from now on.

The short exact sequence $0\rightarrow \mathbf{G}_{\mathrm{m}}\rightarrow X'\rightarrow \widehat{G}\rightarrow 0$ splits canonically when $X = G$, hence the comparison map can actually be lifted to the composition
$$\phi : \mathrm{H}^2(k, \widehat{G})\longrightarrow\mathrm{H}^2(k, \mathcal{Mor}(G, \mathbf{G}_{\mathrm{m}}))\longrightarrow\mathrm{Br}(G)$$
which lands in $\mathrm{Br}(G)$ instead of $\mathrm{Br}(G)/\mathrm{im}\,\mathrm{Br}(k)$.
The final lemma of this section, to be used~in section $4$, is the following:

\begin{lem}\label{lembmloc}
Let $k_v$ be a completion of $k$. Then the local part at $v$ of the Brauer-Manin pairing commutes with the cup product (composed with $G\otimes\widehat{G}\rightarrow\mathbf G_{\mathrm m}$), in the following sense:
\begin{center}\begin{tikzcd}[column sep=tiny]
    \mathrm{H}^2(k_v,\widehat{G})\arrow[d, "\phi", swap]
\hspace{-5pt}&\hspace{-10pt}{\times}\hspace{-10pt}&\hspace{-5pt}
    \mathrm{H}^0(k_v,G)
    \arrow[d, equal]\arrow[rrrrrrrrrrrr, "\smile"]
    &&&&&&&&&&&& \mathrm{H}^2(k_v,\mathbf G_{\mathrm m})
    \arrow[d, equal]\arrow[rrrr, "\mathrm{inv}_v"]
    &&&& \mathbf Q/\mathbf Z\arrow[d, equal]\\
    \mathrm{Br}(G_{k_v})
\hspace{-5pt}&\hspace{-10pt}{\times}\hspace{-10pt}&\hspace{-5pt}
    G(k_v)
    \arrow[rrrrrrrrrrrr, "\mathrm{eval}"]
    &&&&&&&&&&&& \mathrm{Br}(k_v)
    \arrow[rrrr, "\mathrm{inv}_v"]
    &&&& \mathbf Q/\mathbf Z
\end{tikzcd}\end{center}
\end{lem}
\begin{prf}
Unpacking the construction of the new comparison map $\phi$, we see that it is sufficient to show commutativity of the following diagram for a fixed element $g\in G(k_v)$
\begin{center}\begin{tikzcd}[column sep=tiny]
    \mathrm{H}^2(k_v,\widehat{G})\arrow[d]\arrow[rrrrrrrrrrrr, "\;- \;\smile\; g"]
    &&&&&&&&&&&& \mathrm{H}^2(k_v,\mathbf G_{\mathrm m})
    \arrow[d, equal]\\
    \mathrm{H}^2(k_v, \mathcal{Mor}(G, \mathbf{G}_{\mathrm{m}}))\arrow[d]\arrow[rrrrrrrrrrrr, "- \;\circ\; g\;"]
    &&&&&&&&&&&& \mathrm{H}^2(k_v, \mathcal{Mor}(k_v, \mathbf{G}_{\mathrm{m}}))
    \arrow[d, equal]\\
    \mathrm{Br}(G_{k_v})\arrow[rrrrrrrrrrrr, "g^*"]
    &&&&&&&&&&&& \mathrm{Br}(k_v)
\end{tikzcd}\end{center}
where $\mathcal{Mor} = \mathcal{Mor}_{k_v}$ is the sheaf of morphisms of schemes over $k_v$. The top row is given by the cup product with $g$, the middle rows by precomposition of morphisms with $g$ (Yoneda pairing), and the final row by evaluation of the Brauer group at $g$. 

To show commutativity, first fix an element $g\in \mathrm{H}^0(k_v,G)$, seen as a morphism $\mathrm{Spec}\,k_v\rightarrow G$. By definition, the cup product in degree $0$ is given simply by composition, hence we have the following commutative square of sheaf morphisms:
\begin{center}\begin{tikzcd}
    \widehat{G}\arrow[d, hook]\arrow[r, "\;- \;\smile\; g"]
    & \mathbf G_{\mathrm m}\arrow[d, equal]\\
    \mathcal{Mor}(G, \mathbf{G}_{\mathrm{m}})\arrow[r, "- \;\circ\; g\;"]
    & \mathcal{Mor}(k_v, \mathbf{G}_{\mathrm{m}})
\end{tikzcd}\end{center}\noindent
We get commutativity of the first two rows in the diagram by taking second cohomology groups.
Next, pulling back by $g$ induces the following map of Leray spectral sequences,
\begin{center}
\;\;\;\;\;\;\;\;\;\;\;\;\;\;\;\begin{tikzcd}
\!\!\!\!\!\!\!\!\!\!\!\!\!\!\!\!\!\!\!\!\!\!\!\!\!
    \,E^{p,q}_2\coloneqq \mathrm{H}^p(k_v,\mathcal{Mor}^q(G,\mathbf G_{\mathrm m}))\arrow[r, Rightarrow]\arrow[d] & R^{p+q}f_*(\mathbf G_{\mathrm m,\,G_{k_v}})\arrow[d]\\
\!\!\!\!\!\!\!\!\!\!\!\!\!\!\!\!\!\!\!\!\!\!\!\!\!
    'E^{p,q}_2\coloneqq \mathrm{H}^p(k_v,\mathcal{Mor}^q(k_v,\mathbf G_{\mathrm m}))\arrow[r, Rightarrow] & R^{p+q}\mathrm{id}_{k_v,*}(\mathbf G_{\mathrm m})
\end{tikzcd}\end{center}\noindent
where $f : G_{k_v}\rightarrow k_v$ is the structure morphism. 
The ladder of their associated exact sequences immediately gives us commutativity of the bottom half of the diagram (where the edge maps are vertical). It remains to see that the map $\mathrm{H}^2(k_v, \mathcal{Mor}(k_v, \mathbf{G}_{\mathrm{m}}))\rightarrow\mathrm{Br}(k_v)$ induced by $'E_r$ is indeed the canonical identification, but this holds because the spectral sequence $'E_r$ is actually constant for $r\geq 2$.
\end{prf}
\section{Global duality and obstruction to the Hasse principle}

Let $X$ be a homogeneous space of a commutative affine algebraic group $G$ over $k$. We want~to show that the Brauer-Manin obstruction is the only obstruction to the Hasse principle on $X$. 

\begin{rem}
We claim that $X$ is also a \textit{principal} homogeneous space of a commutative affine algebraic group over $k$, i.e.\ a torsor. To see this, consider a finite field extension $k'/k$ over which there exists a point $x\in X(k')$. We write $\overline{H}\subseteq G_{k'}$ for the stabilizer of $x$. The $k$-forms $G$ and $X$, respectively of $\overline{G}\coloneqq G_{k'}$ and $\overline{X}\coloneqq X_{k'}$, define fppf descent data over $k'\otimes_k k'$ of~the~forms:
$$\varphi_G : \mathrm{pr}_1^*\overline{G}\rightarrow\mathrm{pr}_2^*\overline{G}
\;\;\;\textrm{ and }\;\;\;
\varphi_X : \mathrm{pr}_1^*\overline{X}\rightarrow\mathrm{pr}_2^*\overline{X}$$
Then $\varphi_G(\mathrm{pr}_1^*\overline{H})\subseteq\mathrm{pr}_2^*\overline{G}$ is the stabilizer of the point $\varphi_X(x\circ\mathrm{pr}_1)$. However, it must also be the stabilizer $\mathrm{pr}_2^*\overline{H}$ of $(x\circ\mathrm{pr}_2)$ since $G$ is commutative (indeed, these stabilizers must agree over~an fppf cover $R$ of $k'\otimes_k k'$ over which there exists $g\in G(R)$ such that $(x\circ\mathrm{pr}_1).g = (x\circ\mathrm{pr}_2)$).

This shows that $\varphi_G$ restricts to a descent datum $\varphi_H$ on $\overline{H}$ along which the inclusion $\overline{H}\subseteq\overline{G}$ descends to $k$. In particular, $X$ is a principal homogeneous space of $G/H$ for the associated $k$-form $H$ of $\overline{H}$.
\end{rem}

In view of this remark, we may assume from now on that $X$ is a torsor of $G$ over $k$. To study the Hasse principle, suppose that $X$ is \textit{locally trivial}, which is to say that $X(k_v)\neq\varnothing$ holds for all $v$. Equivalently, the class $[X]\in\mathrm{H}^1(k, G)$ lies in $\Sh^1(G)$, since a torsor is trivial over some base if and only it has a point over it. In fact, more is true, as $X(\mathbf A)\neq\varnothing$.
Indeed, the class $[X]$ of a locally trivial torsor is in the kernel of a composition of maps
$$\mathrm{H}^1(k, G)\rightarrow\mathrm{H}^1(\mathbf A, G)\rightarrow\prod\nolimits'_v\mathrm{H}^1(k_v, G)\hookrightarrow\prod\nolimits_v\mathrm{H}^1(k_v, G)$$
where $\prod'$ denotes a restricted product in the appropriate sense. The middle map is a bijection by Proposition 2.17 in \cite{CesPT}, therefore $[X]$ vanishes also in $\mathrm{H}^1(\mathbf A, G)$. It follows that there exists an adelic point $(P_v)\in X(\mathbf A)\subseteq\prod_v X(k_v)$. The Hasse principle then holds if and only if $[X] = 0$ in $\mathrm{H}^1(k, G)$, and we may apply the Brauer-Manin pairings constructed earlier.

In this section, we show that the only obstruction to the Hasse principle on $X$ is the~one~given by the group $\Be(X)$, as defined in the previous section. With this aim, we prove the compatibility (via the comparison map $\Phi_X$) of the obstruction with the global Poitou-Tate pairing 
$$\langle -,-\rangle_{PT} : \Sh^1(G)\times\Sh^2(\widehat{G})\longrightarrow\mathbf Q/\mathbf Z$$
which is defined using Čech cohomology in \cite{RosTD}, \textsection 5.13, for all commutative affine algebraic groups $G$. More precisely, we show that the composition $\Sh^2(\widehat{G})\rightarrow\Be(X)\rightarrow\mathbf Q/\mathbf Z$ of the map $\Phi_X$ and the obstruction map  agrees up to sign with $\langle [X],-\rangle_{PT}$.
Thus, when the Brauer-Manin obstruction vanishes, the nondegeneracy of the Poitou-Tate pairing (\cite{RosTD},~Theorem~1.2.10) immediately gives that the class $[X]\in\Sh^1(G)$ is trivial. Equivalently, $X$ has a $k$-point.
\smallskip

Before recalling Rosengarten's construction of this pairing, let us first note that we will also use Čech cohomology to prove this agreement and now check that we can indeed do so:
\begin{prop}\label{propcechcoh}
For a global field $k$, the canonical homomorphisms $\mathrm{\check H}^2(k,\mathbf G_{\mathrm{m}})\rightarrow\mathrm{H}^2(k,\mathbf G_{\mathrm{m}})$, $\mathrm{\check H}^2(k,X')\rightarrow\mathrm{H}^2(k,X')$ and $\mathrm{\check H}^2(k,\widehat{G})\rightarrow\mathrm{H}^2(k,\widehat{G})$ are isomorphisms. In case $X$ is locally trivial, the same is true over all completions $k_v$ of $k$. All three Čech cohomology groups can be computed using the (not generally fppf) single-element cover $\mathrm{Spec}(\overline{k})\rightarrow\mathrm{Spec}(k)$.
\end{prop}
\begin{proof}
The last statement is Proposition 2.9.3 in \cite{RosTD}, and the first and third maps are isomorphisms over any field by Propositions 2.9.6 and 2.9.9, ibid. The second map is thus also an isomorphism if the short exact sequence 
$0\rightarrow \mathbf{G}_{\mathrm{m}}\rightarrow X'\rightarrow \widehat{G}\rightarrow 0$ splits, which happens over all $k_v$ when $X$ is locally trivial.
It remains only to prove bijectivity of the second map~over~$k$:\;\;\;\;

The same exact sequence splits over $\overline{k}$, hence $\mathrm{\check C}^{\bullet}(\overline{k}/k,-)$ gives a short exact sequence of Čech complexes, and thus a long exact sequence in Čech cohomology. This sequence is compatible with the long exact sequence in derived-functor cohomology 
\vspace{-7pt}
\begin{center}\begin{tikzcd}
    \mathrm{\check H}^1(\overline{k}/k,\widehat{G})\arrow[r]\arrow[d, "\rotatebox{90}{$\sim$}", pos=0.4] & \mathrm{\check H}^2(\overline{k}/k,\mathbf G_{\mathrm{m}})\arrow[r]\arrow[d, "\rotatebox{90}{$\sim$}", pos=0.4] & \mathrm{\check H}^2(\overline{k}/k,X')\arrow[r]\arrow[d] & \mathrm{\check H}^2(\overline{k}/k,\widehat{G})\arrow[r]\arrow[d, "\rotatebox{90}{$\sim$}", pos=0.4] & 0\\
    \mathrm{H}^1(k,\widehat{G})\arrow[r]& \mathrm{H}^2(k,\mathbf G_{\mathrm{m}})\arrow[r] & \mathrm{H}^2(k,X')\arrow[r] & \mathrm{H}^2(k,\widehat{G})\arrow[r] & 0
\end{tikzcd}\end{center}
\vspace{-7pt}
by the delta-functoriality shown in \cite{RosTD}, Proposition F.2.1. The above diagram has exact rows since $\mathrm{\check H}^3(k,\mathbf G_{\mathrm{m}}) = 0$ by Lemma 5.11.1, ibid. The vertical map on the left is known to be an isomorphism in general. We finish with the 5-lemma.
\end{proof}

The global Poitou-Tate pairing $\Sh^1(G)\times\Sh^2(\widehat{G})\rightarrow\mathbf Q/\mathbf Z$ takes a pair $([X],A)$ to $\langle [X],A\rangle_{PT}$ that is defined in \cite{RosTD}, \textsection 5.13, as follows: Take Čech cocycles $\alpha\in \mathrm{\check Z}^1(k,G)$ and $\alpha'\in \mathrm{\check Z}^2(k,\widehat{G})$ representing $[X]$, $A$, respectively. For all $v$, there exist cochains $\beta_v\in \mathrm{\check C}^0(k_v,G)$ 
with $\alpha_v = d\beta_v$. 
Moreover, since $\mathrm{\check H}^3(k,\mathbf G_\mathrm{m}) = 0$, there is a cochain $h\in\mathrm{\check C}^2(k,\mathbf G_\mathrm{m})$ such that $dh = \alpha\!\smallsmile\!\alpha'$.~Then $dh_v = d(\beta_v\!\smallsmile\!\alpha'_v)$, defining elements $[(\beta_v\!\smallsmile\!\alpha'_v)-h_v]\in\mathrm{\check H}^2(k_v,\mathbf G_\mathrm{m})$ (cf.\ Remark \ref{rempt} below). The pairing $\langle [X],A\rangle_{PT}$ is well-defined as the sum of the corresponding invariants in $\mathbf Q/\mathbf Z$, i.e. the resulting sum is finite and its value independent of all choices made in the construction.
\smallskip

Let $\midotimes^4\overline{k} = \overline{k}\otimes_k\overline{k}\otimes_k\overline{k}\otimes_k\overline{k}$. By Proposition \ref{propcechcoh}, the above cup product $\alpha\!\smallsmile\!\alpha'$ is, explicitly, the image in $\mathbf G_\mathrm{m}(\midotimes^4\overline{k})$ of $\alpha\circ\mathrm{pr}_{1,2}\in G(\midotimes^4\overline{k})$ via the map induced on $\midotimes^4\overline{k}$-points by~the~morphism $\alpha'\circ\mathrm{pr}_{2,3,4}\in\widehat{G}(\midotimes^4\overline{k})$, where $\mathrm{pr}_J$ denotes respective projections from $(\mathrm{Spec}\,\overline{k})^4$ to $(\mathrm{Spec}\,\overline{k})^{|J|}$.\;\;\;\;\;\;\;\;

\begin{thm}\label{thmmainglobal}
Let $G$ be an affine commutative algebraic group and $X$ a locally-trivial $G$-torsor over $k$. If $(P_v)\in X(\mathbf A)$ is arbitrary, and if $\Phi_X : \Sh^2(\widehat{G})\rightarrow \Be(X)$ is the comparison map, then
$$\langle [X],A\rangle_{PT} = - \langle (P_v),B\rangle_{BM}$$ 
holds in $\mathbf Q/\mathbf Z$, for any $A\in\Sh^2(\widehat{G})$ and representative $B\in\mathrm{Br}(X)$ of $\Phi_X(A)$.

Consequently, for any homogeneous space $X$ of $G$ over $k$, the Brauer-Manin obstruction given by $\Be(X)$ is the only obstruction to the Hasse principle on $X$.
\end{thm}

\begin{prf}
By Lemma \ref{lembmglob} and Proposition \ref{propcechcoh}, it suffices to show that $-\langle [X],-\rangle_{PT}$ is compatible with the connecting homomorphism coming from the following diagram in Čech cohomology:
\vspace{-17pt}

\begin{center}\begin{equation}\label{diagthmglob}
\begin{tikzcd}
    0\arrow[r] & \mathrm{\check H}^2(k,\mathbf G_{\mathrm{m}})\arrow[r]\arrow[d] & \mathrm{\check H}^2(k,X')\arrow[r]\arrow[d] & \mathrm{\check H}^2(k,\widehat{G})\arrow[r]\arrow[d] & 0\\
    0\arrow[r] & \prod_v\mathrm{\check H}^2(k_v,\mathbf G_{\mathrm{m}})\arrow[r]
    &  \prod_v\mathrm{\check H}^2(k_v,X')\arrow[r] &  \prod_v\mathrm{\check H}^2(k_v,\widehat{G})\arrow[r] & 0\\
\end{tikzcd}\end{equation}\end{center}
\vspace{-25pt}

\noindent
Indeed, both sides of the identity are given as images under $\sum\mathrm{inv}_v : \bigoplus_v \mathrm{\check H}^2(k_v,\mathbf G_\mathrm{m})\rightarrow\mathbf Q/\mathbf Z$. If we carefully handle the choices in both constructions, we can get the same classes on both sides in all $\mathrm{\check H}^2(k_v,\mathbf G_\mathrm{m})$, which will end the proof. Concerning notation, we write $\midotimes^3\overline{k}$ for $\overline{k}\otimes_k\overline{k}\otimes_k\overline{k}$ and similarly with $\overline{k_v}$ over $k_v$.

Fix algebraic closures of $k$ and all $k_v$. Fix a point $\overline{b}$ in $X(\overline{k})$. We write $\mathrm{pr}_i : (\mathrm{Spec}\,\overline{k})^2\rightarrow \mathrm{Spec}\,\overline{k}$ for $i = 1,2$. There is a unique element $\alpha\in G(\overline{k}\otimes\overline{k})$ such that $(\overline{b}\circ\mathrm{pr}_1)+\alpha = \overline{b}\circ\mathrm{pr}_2$ in $X(\overline{k}\otimes\overline{k})$; then $\alpha\in\mathrm{\check Z}^1(k,G)$, it represents the class $[X]\in\mathrm{\check H}^1(k,G)$. 
Next, fix points $b_v\in X(k_v)$ for all $v$. They give $G_{k_v}\mathrm{-Sch}/{k_v}$ isomorphisms $B_v : G_{k_v}\xrightarrow{\sim}X_{k_v}$ by taking $0$ to $b_v$. Finally, we can define $\beta_v\coloneqq\overline{b}|_{\overline{k_v}} - b_v|_{\overline{k_v}}\in G({\overline{k_v}})$ so that $d\beta_v = \alpha|_{\overline{k_v}} - 0 = \alpha_v$.

We can assume that the element $\alpha'\in \mathrm{\check Z}^2(k,\widehat{G})$ representing $A$ has a preimage $\gamma'\in \mathrm{\check Z}^2(k,X')$. Thus $\gamma'\in X'(\midotimes^3\overline{k})$, hence it's by definition a map $\gamma' : X_{\midotimes^3\overline{k}}\rightarrow\mathbf G_{\mathrm{m},\midotimes^3\overline{k}}$ such that $\widetilde{\gamma'} = \alpha'$, in the sense of Definition \ref{defxprim}. This defines elements $\gamma'_v\circ B_v\in\mathrm{\check Z}^2(k_v,G')$ (since $b_v\in X(k_v)\subseteq X(\midotimes^3\overline{k_v})$), where $B_v$ allows us to replace the bottom row of \eqref{diagthmglob} by the product of split exact sequences:
$$0\longrightarrow \mathrm{\check H}^2(k_v,\mathbf G_{\mathrm{m}})\longrightarrow \mathrm{\check H}^2(k_v, G')\longrightarrow \mathrm{\check H}^2(k_v,\widehat{G})\longrightarrow 0$$
Because $A\in\Sh^2(\widehat{G})$, we see that $[\alpha'_v] = 0\in\mathrm{\check H}^2(k_v,\widehat{G})$. This is the image of the class $[\gamma'_v\circ B_v]$, which is hence naturally in $\mathrm{\check H}^2(k_v,\mathbf G_{\mathrm{m}})$.

The split-exactness allows us to consider both $\mathbf G_{\mathrm{m}}(\midotimes^3\overline{k_v})$ and $\widehat{G}(\midotimes^3\overline{k_v})$ inside $G'(\midotimes^3\overline{k_v})$. We claim that we can choose $h\in\mathrm{\check C}^2(k,\mathbf G_{\mathrm{m}})$ such that $dh = \alpha\smallsmile\alpha'$ and
$$(\beta_v\!\smallsmile\!\alpha'_v)-h_v = \alpha'_v - \gamma'_v\circ B_v$$
in $\mathrm{\check Z}^2(k_v,G')$. This claim implies the statement of the theorem, because $[\alpha'_v] = 0$ and thus the two sides correspond to the two sides of the desired identity.

From this point onwards, we fix inclusions $\overline{k}\lhook\joinrel\xrightarrow{\;\mathrm{id}\otimes 1\;} \overline{k}\otimes\overline{k}\lhook\joinrel\xrightarrow{\;\mathrm{id}\otimes 1\otimes 1\;}
\midotimes^3\overline{k}\lhook\joinrel\xrightarrow{\;\mathrm{id}\otimes 1\otimes 1\otimes 1\;}
\midotimes^4\overline{k}$ and similarly for $\overline{k_v}$, so that an expression such as $\overline{b}|_{\overline{k}\otimes\overline{k}}$ is well-defined.
The element $\overline{b}$ gives a map $X'_{\overline{k}}\rightarrow\mathbf G_{\mathrm{m}, \overline{k}}$ by evaluation. We define $h\coloneqq\gamma'(\overline{b}) = \gamma'(\overline{b}|_{\midotimes^3\overline{k}})$ and check the desired properties in order. First, a straightforward computation (in which, again, $\mathrm{pr}_J$ denotes respective projections from $(\mathrm{Spec}\,\overline{k})^N$ to $(\mathrm{Spec}\,\overline{k})^{|J|}$) gives, because $d\gamma' = 0$, that:
\vspace{-15pt}

\begin{align*}
    dh = d(\gamma'(\overline{b})) - (d\gamma')(\overline{b}) &= \sum_{i=1}^4 (-1)^i \big(\gamma'(\overline{b}\circ\mathrm{pr}_1)\big)\circ\mathrm{pr}_{1,\ldots,\widehat{i},\ldots,4} - \sum_{i=1}^4 (-1)^i (\gamma'\circ\mathrm{pr}_{1,\ldots,\widehat{i},\ldots,4})(\overline{b}\circ\mathrm{pr}_1)\\
    &= (\gamma'\circ\mathrm{pr}_{2,3,4})(\overline{b}\circ\mathrm{pr}_2) - (\gamma'\circ\mathrm{pr}_{2,3,4})(\overline{b}\circ\mathrm{pr}_1)\\
    &= (\widetilde{\gamma'}\circ\mathrm{pr}_{2,3,4})(\overline{b}\circ\mathrm{pr}_2 - \overline{b}\circ\mathrm{pr}_1) = (\alpha'\circ\mathrm{pr}_{2,3,4})(\alpha\circ\mathrm{pr}_{1,2}) = \alpha\smallsmile\alpha'
\end{align*}
Second, we have for any algebra 
$R$ over $\midotimes^3\overline{k_v}$ and any $g\in G(R)$, that:
\begin{align*}
    \big(\alpha'_v-\gamma'_v\circ B_v\big)(g) &= \alpha'_v(g) - \gamma'_v(g+b_v)\\
    &= \alpha'_v(g) - \big(\widetilde{\gamma'_v}(g+b_v-\overline{b}) + \gamma'_v(\overline{b})\big)
    = \alpha'_v(\overline{b}-b_v) - \gamma'_v(\overline{b})
    = \alpha'_v(\beta_v) - h_v
\end{align*}
Note that the resulting cocycle is constant; confirmation that both sides are inside $\mathrm{\check Z}^2(k_v,\mathbf G_{\mathrm{m}})$. Finally, $\alpha'_v(\beta_v|_{\midotimes^3\overline{k}}) = \beta_v\!\smallsmile\!\alpha'_v$ since $\beta_v$ is a $0$-cocycle, which completes the calculation.
\end{prf}

Before ending this section, we return to a small detail omitted in the above proof:

\begin{rem}\label{rempt}
First, $\,\widehat{\!\widehat{G}}\cong G$ by double duality (Proposition 2.4.3 in \cite{RosTD}). The construction of the Poitou-Tate pairing in \cite{RosTD}, \textsection 5.13, is stated in a generality including any $\mathcal F\in\{G,\widehat{G}\}$ and there the bilinear map $\langle -,-\rangle_{PT} : \Sh^2(\mathcal F)\times\Sh^1(\widehat{\mathcal F})\longrightarrow\mathbf Q/\mathbf Z$ is defined by the sum:
$$\langle [\alpha],[\alpha']\rangle_{PT} = \sum\nolimits_v\mathrm{inv}_v(c_v)
\;\;\textrm{ where }\;\;c_v = \big[(\alpha_v\!\smallsmile\!\beta'_v)-h_v\big] = \big[(\beta_v\!\smallsmile\!\alpha'_v)-h_v\big]$$
Here the cochains $\beta_v\in \mathrm{\check C}^1(k_v,\mathcal F)$ (resp.\ $\beta_v'\in \mathrm{\check C}^0(k_v,\widehat{\mathcal F})$) satisfy $\alpha_v = d\beta_v$ (resp.\ $\alpha'_v = d\beta'_v$)~and the cochain $h\in\mathrm{\check C}^2(k,\mathbf G_\mathrm{m})$ satisfies $dh = \alpha\!\smallsmile\!\alpha'$. The two classes in the definition of $c_v$ agree since the difference $(\alpha_v\!\smallsmile\!\beta'_v)-(\beta_v\!\smallsmile\!\alpha'_v) = \mathrm d(\beta_v\!\smallsmile\!\beta'_v)$ is a coboundary.

In the discussion preceding Theorem \ref{thmmainglobal}, it was convenient for us to introduce the Poitou-Tate pairing with the reverse order of terms in the cup product, effectively redefining~$c_v$~to~equal
$$c'_v = \big[(\beta'_v\!\smallsmile\!\alpha_v)-h'_v\big] = -\big[(\alpha'_v\!\smallsmile\!\beta_v)-h'_v\big]$$
for $h'\in\mathrm{\check C}^2(k,\mathbf G_\mathrm{m})$ with $dh' = \alpha'\!\smallsmile\!\alpha$ (these two classes are again equal; the sign change comes from the graded Leibniz rule for the cup product). We claim that this reversal has no effect on the resulting pairing. To see this, consider the \textit{first higher cup product} originally introduced in \cite{Ste47} and given on Čech cochains over a $k$-scheme $S$ by:
$$\mathbin{\smallsmile_1} \,:\;\mathrm{\check C}^m(S,\mathcal F)\times\mathrm{\check C}^n(S,\widehat{\mathcal F})\longrightarrow\mathrm{\check C}^{m+n-1}(S,\mathbf G_\mathrm{m})$$
$$u\mathbin{\smallsmile_1}v\;\coloneqq\;\sum\nolimits_{i=0}^{m-1}\,(-1)^{(m-i)(n+1)}\left(v\circ\mathrm{pr}_{i+1,\,\ldots\,,\,i+n+1}\right)\left(u\circ\mathrm{pr}_{1,\,\ldots\,,\,i+1,\,i+n+1,\,\ldots\,,\,m+n}\right)$$
This is a functorial bilinear operation which moreover satisfies the identity
$$(u\smallsmile v)-(-1)^{mn}(v\smallsmile u) = (-1)^{m+n}\big({-}\mathrm d(u\mathbin{\smallsmile_1}v)+(\mathrm du\mathbin{\smallsmile_1}v)+(-1)^m(u\mathbin{\smallsmile_1}\mathrm dv)\big)$$
as shown in Theorem 5.1 of \cite{Ste47}. In our situation this gives (for $m = 2$ and $n = 0$)
\begin{align*}
    c_v - c'_v &= \big[(\alpha_v\!\smallsmile\!\beta'_v)-h_v\big] - \big[(\beta'_v\!\smallsmile\!\alpha_v)-h'_v\big]\\
    &= \big[{-}\mathrm d(\alpha_v\mathbin{\smallsmile_1}\beta'_v)+(0\mathbin{\smallsmile_1}\beta'_v)+(-1)^2(\alpha_v\mathbin{\smallsmile_1}\mathrm d\beta'_v)-h_v+h'_v\big] = \big[(\alpha_v\mathbin{\smallsmile_1}\alpha'_v)-h_v+h'_v\big]
\end{align*}
in $\mathrm{\check H}^2(k_v,\mathbf G_\mathrm{m})$. Another application of the same identity (for $m = 2$ and $n = 1$) shows that
$$\mathrm d\varepsilon \,=\, \mathrm d(\alpha\mathbin{\smallsmile_1}\alpha')-(\alpha\smallsmile\alpha')+(\alpha'\smallsmile\alpha) \,=\, (0\mathbin{\smallsmile_1}\alpha')+(-1)^2(\alpha\mathbin{\smallsmile_1}0) \,=\, 0$$
for $\varepsilon\coloneq(\alpha\mathbin{\smallsmile_1}\alpha')-h+h'$ and thus $\varepsilon\in\mathrm{\check Z}^2(k,\mathbf G_\mathrm{m})$. In particular, $c_v-c'_v = [\varepsilon]_v$ for all $v$, hence summing over these differences yields $0$ in $\mathbf Q/\mathbf Z$ by the Brauer-Hasse-Noether theorem.

\end{rem}
\section{Obstructions to weak and strong approximation}

Let $X$ be a torsor of a commutative affine group $G$ over $k$. In the study of approximation by rational points on $X$, we may suppose $X(k)\neq\varnothing$ and thus $X\simeq G$ over $k$. From now on, we let $X = G$ without loss of generality. 
Let $S\subseteq\Omega_f$ be a (not necessarily finite) set of finite places of the global field $k$. The aim of this section is to prove that the Brauer-Manin obstruction is the only one to weak (resp. strong) approximation on $G$ with respect to $\Omega_f$ (resp. to $S\neq\Omega$).\;\;\;

For this, define topological groups $\mathbf P_v$ as $G(k_v)$ if $v\in\Omega_f$ and $\mathrm{\widehat H}^0(k_v,G)$ if $v\in\Omega_\infty$, so that there is a local duality isomorphism $\mathbf P_{v,}\pro\rightarrow\mathrm{H}^2(k_v,\widehat{G})^D$ by \cite{RosTD}, Theorem 1.2.3~+~Appendix~G. We let $\mathbf P^S\coloneqq G(\mathbf A^{S\,\sqcup\,\Omega_\infty})\times \prod_{v\in\Omega_\infty}\mathbf P_v$ and also $\mathbf P\coloneqq \mathbf P^\varnothing$.
The following exact sequence is split,
$$0\longrightarrow \mathbf P^S\longrightarrow \mathbf P\longrightarrow G(\mathbf A_S)\longrightarrow 0$$
thus remains exact after taking its profinite completion, so we consider the following diagram,
\vspace{-17pt}

\begin{center}\begin{equation}\label{eq1}
\hspace{-20pt}\begin{tikzcd}
    & & 0\arrow[d] & 0\arrow[d]\\
    & & \mathbf P^S\ppro\arrow[r, "f^S"]\arrow[d, "i^S"] & \big(\mathrm{H}^2(k,\widehat{G})/\Sh_S^2(\widehat{G})\big)^*\arrow[d, "j"]\\
    0\arrow[r] & G(k)\pro\arrow[r]\arrow[d, equal] & \mathbf P\pro\arrow[r, "f"]\arrow[d, "p_S"] & \mathrm{H}^2(k,\widehat{G})^*\arrow[r]\arrow[d, "q"] & \Sh^2(\widehat{G})^*\arrow[r]\arrow[d, equal] & 0\\
    & G(k)\pro\arrow[r] & G(\mathbf A_S)\pro\arrow[r, "f_S"]\arrow[d] & \Sh_S^2(\widehat{G})^*\arrow[r]\arrow[d] & \Sh^2(\widehat{G})^*\arrow[r] & 0\\
    & & 0 & 0
\end{tikzcd}\hspace{-15pt}
\end{equation}\end{center}
\vspace{-3pt}

\noindent
in which the middle row is part of the exact Poitou-Tate sequence from \cite{RosTD}, Theorem~1.2.9 (and Appendix G). In particular, $f$ is defined as the sum over the local duality~maps.~Note~that the $\mathrm{H}^2$ and $\Sh^2$ groups are discrete torsion (ibid, Lemmas~3.2.1 and 3.5.1), and $A^*$ is canonically isomorphic (as an Abelian group) to the profinite group $A^D$ for any discrete torsion Abelian group $A$. We do not distinguish between them, even~when~taking topology into account.


The map $f_S$ is constructed from $f$, first by inducing a continuous map $G(\mathbf A_S)\rightarrow\Sh_S^2(\widehat{G})^*$ (since the local duality pairings for $v\notin S$ with elements of $\Sh_S^2(\widehat{G})$ are all trivial) and then by taking profinite completions. The existence of $f^S$ such that the entire diagram commutes follows by exactness of the two columns. Note that the maps $f$, $f_S$ and $f^S$ are all continuous.\;\;\;\;\;\;

Our first step towards studying approximation properties of $G$ is to deduce exactness of the bottom row of the diagram (the goal here is to study $G(\mathbf A_S)$; otherwise, the same proof would work without the assumption $S\subseteq\Omega_f$, but for a suitably defined $\mathbf P_S$ replacing $G(\mathbf A_S)$):

\begin{prop}\label{prop1}
The lower long row of diagram (\ref{eq1}) is also exact.
\end{prop}
\begin{prf}
Exactness is obvious everywhere except at $G(\mathbf A_S)\pro$. For any $x_S\in G(\mathbf A_S)\pro$ such that $f_S(x_S) = 0$, we need to show that $x_S = p_S(x)$ for some $x\in G(k)\pro\subseteq \mathbf P\pro$.

If $i_S : G(\mathbf A_S)\pro\rightarrow \mathbf P\pro$ is the natural inclusion, then $f(i_S(x_S))\in \ker(q)$. We claim that $f^S$ is surjective, so that $f(i_S(x_S)) = j(f^S(x^S))$ for some $x^S\in \mathbf P^S\ppro$. Then $x = i_S(x_S)-i^S(x^S)$ is the desired element of $G(k)\pro$. To prove this surjectivity is equivalent to proving exactness of:
$$\mathbf P^S\ppro\xrightarrow{\;f\;\!\circ\;\! i^S\;}
\mathrm{H}^2(k,\widehat{G})^*\xrightarrow{\;\;\;\;\;\;\;\;\;} \Sh_S^2(\widehat{G})^*$$

It is equivalent to say that $\mathbf P^S\ppro$ has dense image in $\ker(\mathrm{H}^2(k,\widehat{G})^*\rightarrow\Sh_S^2(\widehat{G})^*)$ (into which it clearly lands). This is the strategy of Proposition 5.9.1 in \cite{RosTD}, in which the case $S = \varnothing$ is proven in positive characteristic. However, the same proof essentially holds for general $S$ with minimal modification. 
We recall it now, replacing $G(k_v)$ by $\mathbf P_v$ everywhere:

In the original proof, given an arbitrary finite subset $T$ of $\mathrm{H}^2(k, \widehat{G})$ and a map $\phi\in\mathrm{H}^2(k, \widehat{G})^*$ vanishing on $\Sh^2(\widehat{G})$, we want to show that there exists an adelic point $g\in\mathbf P$ such that the image of $g$ in $\mathrm{H}^2(k, \widehat{G})^*$ agrees with $\phi$ on $T$. 
Observe first that $T$ can be replaced, without~loss~of generality, by the finite subgroup it generates in the torsion group $\mathrm{H}^2(k, \widehat{G})$. Then, any~finite~set $S(T)\subset\Omega$ of places of $k$ is chosen such that the map $T/(T\cap \Sh^2(\widehat{G}))\rightarrow\prod_{v\in S(T)}\mathrm{H}^2(k_v, \widehat{G})$ is an injection. The dual map is hence surjective. We obtain $\big(\!\prod_{v\in S(T)}\mathbf P_v\big)\pro = \prod_{v\in S(T)}\mathrm{H}^2(k_v, \widehat{G})^*$~by local duality, and therefore $\prod_{v\in S(T)}\mathbf P_v$ has dense image in the finite group $\big(T/(T\cap \Sh^2(\widehat{G}))\big)^*$. In particular, there must exist a point $g_{S(T)}\in \prod_{v\in S(T)}\mathbf P_v$ with image $\phi|_T$. The desired adelic point $g\in\mathbf P$ is constructed by extending $g_{S(T)}$ by $0$ on all places $v\notin S(T)$. 
This shows that $\mathbf P$ has dense image in $\ker(\mathrm{H}^2(k,\widehat{G})^*\rightarrow\Sh^2(\widehat{G})^*)$.

For general $S$, we are given $\phi\in\mathrm{H}^2(k, \widehat{G})^*$ vanishing on $\Sh_S^2(\widehat{G})$ and we want to find a similar injective map on the modified quotient $T/(T\cap \Sh_S^2(\widehat{G}))$. We simply note that~we~can choose the set $S(T)$ disjoint from our fixed $S$. Indeed, for every $\alpha\in T\setminus\Sh_S^2(\widehat{G})$, we can choose some place $v = v(\alpha)\notin S$ with $\alpha_v\neq 0$, which exists by definition of $\Sh_S^2(\widehat{G})$, and define $S(T)$ as the collection of all $v(\alpha)$. 
The construction of $g\in\mathbf P$ in Rosengarten's above~proof~now gives $g_v = 0$ for all $v\notin S(T)$, so in particular for all $v\in S$. Therefore, $g$ is in $\mathbf P^S$, which is all that we needed to show.
\end{prf}

Next, consider the commutative diagram (in which the $S$-overline denotes closure in $G(\mathbf A_S)$)
\vspace{-33pt}

\begin{center}\begin{equation}\label{eq2}
\hspace{-20pt}\begin{tikzcd}
    0\arrow[r] & \overline{G(k)}^{_S}\arrow[r]\arrow[d, dashed] & G(\mathbf A_S)\arrow[r, "f'_S"]\arrow[d, "\varphi"] & \Sh_S^2(\widehat{G})^*\arrow[r, "g_S"]\arrow[d, equal] & \Sh^2(\widehat{G})^*\arrow[d, equal]\arrow[r] & 0\\
    0\arrow[r] & M\arrow[r] & G(\mathbf A_S)\pro\arrow[r, "f_S"] & \Sh_S^2(\widehat{G})^*\arrow[r, "g_S"] & \Sh^2(\widehat{G})^*\arrow[r] & 0
\end{tikzcd}\hspace{-15pt}
\end{equation}\end{center}
\vspace{-3pt}

\noindent
where $M$ is the image of $G(k)\pro\rightarrow G(\mathbf A_S)\pro$; the bottom row is hence exact. The first vertical map exists because $M = \ker(f_S)$ is closed in $G(\mathbf A_S)\pro$ and $G(k)$ maps into it via the canonical morphism $\varphi : G(\mathbf A_S)\rightarrow G(\mathbf A_S)\pro$.

\begin{prop}\label{prop2}
Suppose that $S\subseteq\Omega_f$ (only significant when $\mathrm{char}\,k = 0$) and $S\neq\Omega$ (only significant when $\mathrm{char}\,k > 0$). The top row of diagram (\ref{eq2}) is then exact.
\end{prop}
\begin{prf}
First, the compositions of adjacent maps are clearly $0$ everywhere. Note that the image of $f'_S = f_S\circ\varphi$ is dense in $\mathrm{im}(f_S) = \mathrm{ker}(g_S)$. Thus, to prove exactness at $\Sh_S^2(\widehat{G})^*$, we only have to show that $\mathrm{im}(f'_S)\simeq G(\mathbf A_S)/\overline{G(k)}^{_S}$ is compact. We will call this ``Claim 1''.

Apart from that, we only need to prove exactness at $G(\mathbf A_S)$, which is equivalent to saying that the inclusion $\overline{G(k)}^{_S}\hookrightarrow \ker(f_S\circ\varphi) = \varphi^{-1}(\ker (f_S)) = \varphi^{-1}(M)$ is an equality. The induced map $G(k)\pro\rightarrow G(\mathbf A_S)\pro$ of profinite completions is defined as the composition
$$G(k)\pro = \varprojlim\nolimits_V G(k)/V \xrightarrow{\;\;\;\;\;\;\;\;\;\;\;\;} \varprojlim\nolimits_U G(k)/(G(k)\cap U) \xrightarrow{\;\;\;\;\;\;\;\;\;\;\;\;} \varprojlim\nolimits_U G(\mathbf A_S)/U = G(\mathbf A_S)\pro$$
where the inverse limits are taken over all finite-index open subgroups $V$ of $G(k)$, resp. $U$ of $G(\mathbf A_S)$.
Moreover, the second arrow is injective (because so is $G(k)/(G(k)\cap U) \rightarrow G(\mathbf A_S)/U$ for all subgroups $U$ of $G(\mathbf A_S)$, and inverse limits are left exact), so $M\hookrightarrow G(\mathbf A_S)\pro$ also factors through it, and we have the following diagram,
\begin{center}\begin{tikzcd}
    \overline{G(k)}^{_S}\arrow[r, hook]\arrow[d] & \bigcap_U(G(k)+U)\arrow[r, hook]\arrow[d, dashed] & G(\mathbf A_S)\arrow[d, "\varphi"]\\
    M\arrow[r, hook] & \varprojlim_U G(k)/(G(k)\cap U)\arrow[r, hook] & G(\mathbf A_S)\pro
\end{tikzcd}\end{center}
where both the intersection and the inverse limit are indexed by the family of all finite-index open subgroups $U$ of $G(\mathbf A_S)$. The middle map is given by $(G(k)+U)/U\cong G(k)/(G(k)\cap U)$ for each $U$, and the right square is Cartesian. We claim that the map $\overline{G(k)}^{_S}\hookrightarrow \bigcap_U(G(k)+U)$ is in fact an equality, which then implies that $\overline{G(k)}^{_S} = \varphi^{-1}(M)$. We will call this ``Claim 2''.\;\;\;\;\;\;\;\;\;

Because $S$ contains no Archimedean places, the group $G(\mathbf A_S)$ is both locally compact and totally disconnected (it is a closed subspace of some $(\mathbf A_S)^N$ by definition; see Proposition 2.1 in \cite{ConAd}), hence by Theorem 7.7 in \cite{HR63} it admits a basis of neighborhoods of $0$ consisting of compact open subgroups $W$. Moreover, for any such $W$, the subgroup $G(k)+W$ is open~and of finite index in $G(\mathbf A_S)$ in each the remaining cases of the proposition:
\begin{itemize}
    \item for $\mathrm{char}\,k > 0$ and $S\neq\Omega$ by Theorem 1.3.1 in \cite{Con12} (where the theorem is stated for cofinite $S\neq\Omega$, but follows immediately for smaller $S$ and even noncompact open $W$)
    \item for $\mathrm{char}\,k = 0$ and $S\cap\Omega_\infty = \varnothing$ by Theorem 5.1 in \cite{Brl63} (with a similar remark)
\end{itemize}
This shows Claim 1, since finitely many cosets of any such compact group $W$ cover $G(\mathbf A_S)/G(k)$. Suppose that Claim 2 does not hold and so let $x\in\bigcap_U(G(k)+U)$ be such that there exists a compact open subgroup $W$ of $G(\mathbf A_S)$ for which $(x+W)\cap G(k) = \varnothing$. Now $U\coloneqq G(k)+W$~is~of finite index, but $x\notin G(k)+W = G(k)+U$, a contradiction.
\end{prf}

\begin{prop}\label{prop3}
The following sequence is exact, for $G(k_{\Omega_f}) = \prod_{v\in\Omega_f} G(k_v)$,
\vspace{-15pt}

\begin{center}\begin{equation}\label{eq3}
0\longrightarrow \overline{G(k)}\longrightarrow G(k_{\Omega_f})\xrightarrow{\;\;f'_f\;\;}\Sh_f^2(\widehat{G})^*
\end{equation}\end{center}\noindent
where the product is taken over all places $v$ of $k$, the overline denotes closure in the product topology and the map $f'_f$ is given by the inverse limit of $f'_S$ in (\ref{eq2}) for $S\subseteq\Omega_f$ finite.
\end{prop}
\begin{prf}
Because $\Sh_f^2(\widehat{G})^* = \varprojlim_{\textrm{finite }S\subseteq\Omega_f}\Sh_S^2(\widehat{G})^*$, the image of $G(k)$ is trivial, which extends to its closure. Conversely, for an element $x = (x_v)\in G(k_{\Omega_f})$ outside $\overline{G(k)}$, there exists an open set $U = U_S\times G(k_{\Omega_f\setminus S})$ with $S\subseteq\Omega_f$ finite and $U_S$ open in $G(k_S)$ such that $(x+U)\cap G(k) = \varnothing$. Then $(x_v)_{v\in S}$ is not in the closure of $G(k)$ in $G(k_S)$, so it has nonzero image in $\Sh_S^2(\widehat{G})^*$ by the previous proposition. This map is compatible with the projection from $\Sh_f^2(\widehat{G})^*$, hence $x$ has also nonzero image in $\Sh_f^2(\widehat{G})^*$.
\end{prf}

\begin{rem}\label{reminvlim}
By left-exactness of inverse limits, the above proposition equivalently states that:\hspace{-5pt}
\vspace{-10pt}

$$\overline{G(k)} = \ker\left(G(k_{\Omega_f})\longrightarrow\Sh_f^2(\widehat{G})^*\right) = \varprojlim\nolimits_{\textrm{finite }S\subseteq\Omega_f}\overline{G(k)}^{_S}$$
\end{rem}

Finally, the exact sequences \eqref{eq2} and \eqref{eq3} allow us to deduce the announced statements on Brauer-Manin obstruction to strong and weak approximation. Here, $\overline{G(k)}^{_S}$ is as in \eqref{eq2}:

\begin{thm}\label{thmmainlocal}
Let $S\subseteq\Omega_f$ be a subset of finite places of $k$ such that $S\neq\Omega$. The two sequences 
$$0\longrightarrow \overline{G(k)}^{_S}\longrightarrow  G(\mathbf A_S)\longrightarrow\Be_S(G)^*$$
$$0\longrightarrow \overline{G(k)}\longrightarrow G(k_{\Omega_f})\longrightarrow\Be_f(G)^*$$
are exact sequences of pointed sets, where the $\Be$-groups are as in Section 2.

It follows in particular that, for a torsor $X$ of\, $G$ over $k$, the Brauer-Manin obstruction~to~weak approximation with respect to finite places (resp. strong approximation with respect to $S$) given by $\Be_f(X)$~(resp.~$\Be_S(X)$) is the only obstruction to this approximation on $X$.
\end{thm}
\begin{prf}
The injectivity on the left is clear, as is that the composition of all adjacent maps is $0$. For the other direction it suffices, by Propositions \ref{prop2} and \ref{prop3}, to prove that the rightmost map in both sequences is a factor of the maps to $\Sh_S^2(\widehat{G})^*$ and $\Sh_f^2(\widehat{G})^*$, respectively. Consider
\vspace{-5pt}

\begin{center}\begin{tikzcd}
    G(\mathbf A)\arrow[r, "BM"]\arrow[d, two heads] & \left(\dfrac{\mathrm{Br}(G)}{\mathrm{Br}(k)}\right)^*\arrow[r, "\phi^*"]\arrow[d, two heads] & \mathrm{H}^2(k, \widehat{G})^*\arrow[d, two heads]\\
    G(\mathbf A_S)\arrow[r, "BM"] & \Be_S(G)^*\arrow[r, "\phi^*"] & \Sh_S^2(\widehat{G})^*
\end{tikzcd}\end{center}
\vspace{-3pt}

\noindent
where the left square is commutative by definition of Brauer-Manin obstruction with respect to $\Be_S(G)$, and the right square by functoriality of the comparison map $\phi$ constructed in Section 2. Let $A$ be an arbitrary element of $\mathrm{H}^2(k, \widehat{G})$ and $A_v$ its local images in $\mathrm{H}^2(k_v, \widehat{G})$. The composition of maps in the upper row sends an element $(g_v)\in G(\mathbf A)$ to the function which acts on $A$ as:
\vspace{-13pt}

$$\phi^*(\langle (g_v),-\rangle_{BM})(A) = \langle (g_v),\phi(A)\rangle_{BM} = \sum\nolimits_v\mathrm{inv}_v(g_v^*\,\phi(A_v)) = \sum\nolimits_v\mathrm{inv}_v(g_v\smile A_v)$$
\vspace{-14pt}

\noindent
where we've used Lemma \ref{lembmloc} in the last equality. However, this composition is then by definition exactly the~map $G(\mathbf A)\rightarrow G(\mathbf A)\pro\rightarrow\mathrm{H}^2(k, \widehat{G})^*$ coming from the diagram \eqref{eq1}. In particular, this shows that the composition in the lower row must be $G(\mathbf A_S)\rightarrow G(\mathbf A_S)\pro\rightarrow\Sh_S^2(k, \widehat{G})^*$, exactly as in \eqref{eq2}.

All maps in the lower row are induced by the upper one, hence we immediately get~the~desired factorization, which proves exactness of the first sequence in the statement. We get exactness of the second sequence too~by taking inverse limits over finite subsets $S$.
\end{prf}
\section{Finiteness theorems for $\Sh^2_S(\widehat{G})$ and $\Sh^2_f(\widehat{G})$}

In this section we will prove that the sequence \eqref{eq3} extends to an exact sequence
\vspace{-11pt}

\begin{equation}\label{eq4}
0\longrightarrow \overline{G(k)}\longrightarrow G(k_{\Omega_f})\xrightarrow{\;\;f'_f\;\;} \Sh_f^2(\widehat{G})^*\xrightarrow{\;\;g_f\;\;}  \Sh^2(\widehat{G})^*\longrightarrow 0
\end{equation}
\vspace{-11pt}

\noindent
where $\overline{G(k)}$ is the closure of $G(k)$ in $G(k_{\Omega_f}) = \prod_{v\in\Omega_f} G(k_v)$ with product topology. Surjectivity on the right is obvious. Thanks to Proposition \ref{prop3}, it remains only to show $\smash{\mathrm{im}(f'_f)\supseteq\mathrm{ker}(g_f)}$. This is immediate if $G(k_{\Omega_f})/\overline{G(k)}$ is compact (as in Proposition \ref{prop2}), but we will only deduce this compactness at the end, from the previous statement.

In characteristic $0$, even the larger quotient $\prod_{v\in\Omega} G(k_v)/\overline{G(k)}$ is known to be finite~for~all~the connected affine algebraic groups $G$ (Corollary 3.5(i) in \cite{San81}; where also, in Theorem 5.1, an exact sequence similar to \eqref{eq4} is given, but over all places $v\in\Omega$). We in particular make use of the part of this result for tori proven independently of characteristic in \cite{CTS77} (lemma~below), which itself constitutes most of the proof in characteristic $0$. Therefore, the main content of this section is in positive characteristic, where finiteness of the discussed quotient does not always hold. In particular, the group $\Sh_\omega^2(\widehat{G})$ is not always finite even for connected $G$, and neither is $\Sh_S^2(\widehat{G})$ for a finite~$S$ (Remark \ref{remoes}), although $\Sh^2(\widehat{G})$ always is, by Theorem 1.2.10 in \cite{RosTD}.

\begin{lem}\label{lemsha2torfin}
Let $T$ be a torus over a global field $k$. The quotient $T(k_\Omega)/\overline{T(k)} = \prod\nolimits_v T(k_v)/\overline{T(k)}$ is then finite, hence the sequence \eqref{eq4} is exact for $G = T$. Consequently, $\smash{\Sh_f^2(\widehat{T})}$ is also finite.\;\;
\end{lem}
\begin{proof}
Because $T$ is smooth, we may work with \'etale (and Galois) cohomology. Proposition 18 in \cite{CTS77} shows that the quotient $\prod\nolimits_v T(k_v)/\overline{T(k)}$ is isomorphic to the cokernel of the following map (where $S$ is the flasque torus appearing in a flasque resolution of $T$; see \cite{CTS77})
\vspace{-10pt}

$$\mathrm{H}^1(k,S)\longrightarrow\prod\nolimits_{v\in\Omega}\mathrm{H}^1(k_v,S)$$
\vspace{-11pt}

\noindent
and it claims that both groups are finite. This claim follows from the finiteness of $\prod_{v}\mathrm{H}^1(k_v,S)$. Indeed, $\mathrm{H}^1(k_v,S)$ is finite for all $v$ and any torus $S$ (by \cite{CTS77}, Remark 8, or \cite{Ser97}, III, \textsection 4.3), so it remains only to show why it is also $0$ for almost all $v$ when $S$ is flasque:

If $K/k$ is a finite Galois extension splitting $S$, then the extensions $K_w/k_v$ are cyclic~for~almost all $v\in\Omega$ (the finite unramified places). For such $v$, the torus $S_{k_v}$ is invertible (a direct summand of a quasi-trivial torus) by applying Lemma 2(vii) and Proposition 2 in \cite{CTS77} to the module of characters $M = \mathrm{X}_{L_w}(S_{L_w}) = \mathrm{X}_{L}(S_{L})$. This is enough to conclude $\mathrm{H}^1(k_v,S) = 0$ by Shapiro's lemma and Hilbert's theorem 90.
\end{proof}

\begin{rem}\label{remomf}
Suppose $\mathrm{char}\,k = 0$. A similar argument using Chebotarev's theorem shows that $\Sh_\omega^2(\widehat{T}) = \Sh_f^2(\widehat{T})$ holds for a torus $T$, and then also for all connected $G$. It is already~known, however, that $\smash{\Sh_\omega^2(\widehat{G})}$ is finite (for example by \cite{San81}, Theorem 5.1). All this fails for finite~$G$.
\end{rem}

Next, recall the multiplicative-unipotent decomposition for a commutative affine group:

\begin{lem}\label{lemmatorunip}
Let $G$ be a commutative affine algebraic group over an arbitrary field $K$. Then~$G$ admits a unique algebraic subgroup $H$ of multiplicative type such that there is an exact sequence $0\longrightarrow H\longrightarrow G\longrightarrow U\longrightarrow 0$ with $U$ unipotent. Moreover, if $G$ is connected (resp. smooth), then so are $H$ and $U$. They are the multiplicative and unipotent ``parts'' of $G$, respectively.
\end{lem}
\begin{proof}
These are statements 1.1(a) and 1.4 in \cite{DG70}, IV, \textsection 3.
\end{proof}

There is a natural topology (defined in \cite{RosTD}, 3.3) on $\mathrm{H}^1(K, G)$ for a local field $K$, when $G$ is a group scheme locally of finite type over $K$, which is functorial and $\delta$-functorial (meaning that the long exact sequences coming from such groups have all maps continuous, where $H^2$ is always discrete and $H^0$ comes with its own natural topology). If $G$ is an \textit{almost-torus} (that is, an extension of a finite group by a torus), then $\widehat G$ is representable by a group scheme locally of finite type over $K$ (by \cite{RosTD}, Proposition 2.3.5) and $\mathrm{H}^1(K, \widehat G)$ is hence also equipped with this topology. In particular, short exact sequences~of almost-tori give topological long-exact sequences of their duals.
This topology can be extended to duals of other affine commutative groups. However, the $\delta$-functoriality is then lost (the connecting homomorphisms do~not~have to be continuous). We show that it is possible to recover it in a case of special interest:

\begin{lem}\label{lemmaconnhomcont}
Let $0\longrightarrow H\longrightarrow G\longrightarrow U\longrightarrow 0$ be the multiplicative-unipotent decomposition of a commutative affine algebraic group over a local field $K$. Then the connecting homomorphism $\mathrm{H}^1(K,\widehat{H})\rightarrow \mathrm{H}^2(K,\widehat{U})$ is continuous (equivalently, it has open kernel).
\end{lem}
\begin{proof}
The Frobenius functor $G\rightsquigarrow G^{(p)}$ is exact on algebraic groups over $K$, since it amounts~to a pullback by a map $K\rightarrow K$. Thus there is, for all $n$, a commutative diagram with exact rows,\hspace{-5pt}
\vspace{-6pt}

\begin{center}\begin{tikzcd}
    0\arrow[r] & H\arrow[r]\arrow[d, "\mathrm{Fr}_H^n"] & G\arrow[r]\arrow[d, "\mathrm{Fr}_G^n"] & U\arrow[r]\arrow[d, "\mathrm{Fr}_U^n"] & 0\\
    0\arrow[r] & H^{(p^n)}\arrow[r] & G^{(p^n)}\arrow[r] & U^{(p^n)}\arrow[r] & 0
\end{tikzcd}\end{center}

\noindent
where $\mathrm{Fr}_G : G\rightarrow G^{(p)}$ is the Frobenius morphism. The top row splits over the perfect closure of $k$ by \cite{DG70}, IV, \textsection 3, 1.1(b), hence the bottom row splits over $k$ for large enough $n$. Fix such an $n$. This ensures that the bottom map in the following commutative square is a surjection.
\vspace{-7pt}

\begin{center}\begin{tikzcd}
    \mathrm{H}^1(K,\widehat{G})\arrow[r, "\alpha"] & \mathrm{H}^1(K,\widehat{H})\arrow[r] & \mathrm{H}^2(K,\widehat{U})\\
    \mathrm{H}^1(K,\widehat{G^{(p^n)}})\arrow[r, two heads]\arrow[u] & \mathrm{H}^1(K,\widehat{H^{(p^n)}})\arrow[u, "\beta"]
\end{tikzcd}\end{center}
\vspace{-7pt}

\noindent
We want to prove that the map $\alpha$ has open image. It suffices to show that the image contains an open set, hence it suffices to show that the map $\beta$ is open.

Denote by $I_H$ the infinitesimal multiplicative kernel of $\mathrm{Fr}_H^n$. Looking at the exact sequences
$$0\longrightarrow I_H\longrightarrow H\longrightarrow \mathrm{im(\mathrm{Fr}_H^n)}\longrightarrow 0$$
$$0\longrightarrow \mathrm{im(\mathrm{Fr}_H^n)}\longrightarrow H^{(p^n)}\longrightarrow \mathrm{coker(\mathrm{Fr}_H^n)}\longrightarrow 0$$
we observe first that all the groups involved are of multiplicative type, so in particular almost-tori. Therefore their duals are locally of finite type over $K$. Next, the map $\beta$ is the composition
$$\mathrm{H}^1(K,\widehat{H^{(p^n)}})\longrightarrow\mathrm{H}^1(K,\widehat{\mathrm{im}(\mathrm{Fr}_H^n)})\longrightarrow\mathrm{H}^1(K,\widehat{H})$$
in which the first map is open by \cite{RosTD}, Proposition 3.3.1(vi). The second is also open, by \cite{RosTD}, Proposition 3.3.1(viii), because $\widehat{I_H}$ is a smooth finite group: Indeed, the infinitesimal multiplicative group $I_H$ is a twisted form of a tower of copies of $\mu_p$ (\cite{DG70}, IV, \textsection 3, 5.7), hence its dual is a twisted form of a tower of copies of the smooth group $\mathbf Z/p = \widehat{\mu_p}$.
\end{proof}

\begin{lem}\label{lemmainfgrpcoh}
Let $F$ be a commutative infinitesimal group over a global (function) field $k$. Then:
\begin{enumerate}[$\hspace{0.25 cm}$ a)]
    \item $\mathrm{H}^2(k,\widehat{F}) = 0$
    \item The natural image of $\mathrm{H}^1(k,\widehat{F})$ is dense in $\mathrm{H}^1(k_S,\widehat{F})$, for any finite set $S$ of places of $k$.
\end{enumerate}
\end{lem}
\begin{proof}
Statement (a) can be shown over any field $k$ by \cite{DG70}, IV, \textsection 3, 5.7 and 5.8 and some Galois cohomology, however we give a different proof for $k$ global:
When $S$ is finite, we get via Lemma 6.1.1 in \cite{Con12} that $\Sh^1_S(F) = \Sh^1_S(F')$ for a unique smooth algebraic subgroup $F'$ of $F$ such that $F(K) = F'(K)$ for all separable field extensions $K/k$. However, then $F' = 0$ and $\Sh^1_S(F) = 0$. In particular, $\Sh^2(\widehat{F}) = \Sh^1(F)^* = 0$ by global duality.
Because~$F$~is~infinitesimal, we also have $F(k_v) = 0$ for all $v$. By local duality, $\prod_v \mathrm{H}^2(k_v,\widehat{F}) = 0$, which proves (a).

To show (b), take part $\mathrm{H}^1(k,F)\rightarrow
\mathrm{H}^1(\mathbf A,F)\rightarrow\mathrm{H}^1(k,\widehat{F})^*$ of the Poitou-Tate exact sequence (\cite{RosTD}, Theorem 1.2.9), where $\mathrm{H}^1(k,\widehat{F})^*$ is compact. Restricting to the finite subset $S$ of places of $k$, we conclude exactness of:
$$\Sh^1_S(F)\longrightarrow\mathrm{H}^1(k_S,F)\longrightarrow\mathrm{H}^1(k,\widehat{F})^*$$
Using again that $\Sh^1_S(F) = 0$, the second map is a continuous injection. Taking Pontryagin duals now recovers the map $\mathrm{H}^1(k,\widehat{F})\rightarrow\mathrm{H}^1(k_S,F)^D = \mathrm{H}^1(k_S,\widehat{F})$ (via the local duality stated in \cite{RosTD}, Theorem 1.2.5). It has dense image; otherwise $\mathrm{H}^1(k_S,\widehat{F})/\overline{\mathrm{im}\,\mathrm{H}^1(k,\widehat{F})}$ would have~a nontrivial map to $\mathbf Q/\mathbf Z$, which extends to $\varphi\in\mathrm{H}^1(k_S,\widehat{F})^D$, but maps to $0$ in $\mathrm{H}^1(k,\widehat{F})^*$.
\end{proof}

Before stating the main finiteness result, recall that a unipotent algebraic group (which we~\textit{do not assume} smooth or connected) is called \textit{wound} if it has no subgroup isomorphic to $\mathbf G_{\mathrm a}$.

\begin{thm}\label{thmsha2fin}
Let $G$ be a commutative affine algebraic group over a global field $k$, such that it has no non-finite wound unipotent quotient groups. (This holds in particular if $\mathrm{char}(k) = 0$, or if the unipotent part of $G$ is finite or split, or a product of such groups.)

The group $\Sh_S^2(\widehat{G})$ 
is finite for any finite~set of finite places $S\subseteq\Omega_f$ of $k$. If $G$ is moreover connected, then the group $\Sh_f^2(\widehat{G})$ is also finite (and $\Sh_f^2(\widehat{G}) = \Sh_\omega^2(\widehat{G})$ by Remark \ref{remomf}).
\end{thm}
\begin{rem}\label{remoes}
For $S\subseteq\Omega_f$ finite, $\Sh_S^2(\widehat{G})$ is finite if and only if $G(k_S)/\overline{G(k)}^{_S}$ is (by Proposition \ref{prop2} and finiteness of $\Sh^2(\widehat{G})^*$).
This finiteness is in general false for wound unipotent (commutative) groups. Indeed, if $G$ is wound and of dimension at most $p-2$, then $G(k)$ is finite by \cite{Oes84}, VI.3.1. However, $G(k_v)$ can then very well still be an infinite (Hausdorff) group.

For example, take $k = \mathbf F_p[t]$ for $p > 2$ prime, and the subgroup $U$ of $\mathbf G_{\mathrm a}^2$ given by $y^p-y = tx^p$. Then $U$ is wound of dimension $1\leq p-2$ and the set $G(k_v)$ is infinite by Hensel's lemma (which gives a point $y$ for every $x$ such that $v(tx^p) > 0$) for all places $v$. 
We return to this in Section~6.
\end{rem}

\begin{proof}
Let $S\subseteq\Omega_f$ be a finite set of finite places of $k$. First, consider the connected-\'etale exact sequence $0\longrightarrow C\longrightarrow G\longrightarrow E\longrightarrow 0$ with $C$ connected and $E$ finite. The rows of the diagram
\vspace{-15pt}

\begin{center}\begin{tikzcd}
    & \overline{G(k)}^{_S}\arrow[r]\arrow[d] & \gamma\big(\overline{G(k)}^{_S}\big)\arrow[r]\arrow[d] & 0\\
    C(k_S)\arrow[r] & G(k_S)\arrow[r, "\gamma"] & E(k_S)
\end{tikzcd}\end{center}
\vspace{-5pt}

\noindent
are exact, hence so is the sequence $C(k_S)/\overline{C(k)}^{_S}\longrightarrow G(k_S)/\overline{G(k)}^{_S}\longrightarrow E(k_S)/\gamma\big(\overline{G(k)}^{_S}\big)\longrightarrow 0$. The group $E(k_S)$ is finite, which by the remark above reduces the proof to the case of $G = C$ connected. In particular, since $\Sh_S^2(\widehat{G})\subseteq\Sh_f^2(\widehat{G})$, it suffices to prove that $\Sh_f^2(\widehat{G})$ is finite. 

Let $G$ be connected, and let $H$ and $V$ be the connected multiplicative and unipotent parts of $G$, respectively. Then $H$ is an extension, by a torus $T$, of a connected finite group $F$, thus necessarily infinitesimal and of multiplicative type. Setting $W\coloneqq G/T$ gives exact sequences:
\vspace{-10pt}

$$0\longrightarrow T\longrightarrow G\longrightarrow W\longrightarrow 0$$
$$0\longrightarrow F\longrightarrow W\longrightarrow V\longrightarrow 0$$
\vspace{-14pt}

\noindent
We proceed with the proof in three steps:
\smallskip

\textit{Step 1:} We want to reduce the general proof to the claim that $\gamma_v : \mathrm{H}^2(k,\widehat{W})\rightarrow \mathrm{H}^2(k_v,\widehat{W})$ is injective for all places $v\in\Omega_f$ of $k$. To see why this is sufficient to finish the proof, suppose that the claim is true, write the following diagram with exact rows
\vspace{-5pt}

\begin{center}\begin{tikzcd}
    & \mathrm{H}^2(k,\widehat{W})\arrow[r]\arrow[d, "\gamma_v"] & \mathrm{H}^2(k,\widehat{G})\arrow[r]\arrow[d] & \mathrm{H}^2(k,\widehat{T})\\
    \mathrm{H}^1(k_v,\widehat{T})\arrow[r] & \mathrm{H}^2(k_v,\widehat{W})\arrow[r] & \mathrm{H}^2(k_v,\widehat{G}) & 
\end{tikzcd}\end{center}
\vspace{-5pt}

\noindent
and note first that, by Chebotarev's theorem, we have $\mathrm{H}^1(k_v,\widehat{T}) = 0$ for infinitely many~$v\in\Omega$. (Indeed, $T$ is split by some finite Galois extension $K/k$ and, for every $v$ which is completely split in $K$ and a prime $w|v$ of $K$, we get $K_w = k_v$ and $\mathrm{H}^1(k_v,\widehat{T}) = \mathrm{H}^1(K_w,\widehat{T_K}) = \mathrm{H}^1(K_w,T_K)^D = 0$.) 
Now, any element $B\in\mathrm{ker}(\Sh_f^2(\widehat{G})\rightarrow\Sh_f^2(\widehat{T}))$ is an image of some $A\in \mathrm{H}^2(k,\widehat{W})$. Because $B_v$ is $0$ in $\mathrm{H}^2(k_v,\widehat{G})$ for almost all $v$, we may fix such a place $v\in\Omega_f$ for which also $\mathrm{H}^1(k_v,\widehat{T}) = 0$. This shows that $\gamma_v(A) = 0$ and hence $A = 0$ by injectivity of $\gamma_v$. Finally, we conclude $B = 0$.~It follows that the map $\Sh_f^2(\widehat{G})\rightarrow\Sh_f^2(\widehat{T})$ is an injection, but $\Sh_f^2(\widehat{T})$ is finite by Lemma \ref{lemsha2torfin}.

\textit{Step 2:} To prove that $\gamma_v : \mathrm{H}^2(k,\widehat{W})\rightarrow \mathrm{H}^2(k_v,\widehat{W})$ is injective for all places $v\in\Omega_f$ of $k$, we first suppose that the same is true for $\beta_v : \mathrm{H}^2(k,\widehat{V})\rightarrow \mathrm{H}^2(k_v,\widehat{V})$. Let $v$ be any finite place of $k$ and write the following diagram with exact rows, immediately applying Lemma \ref{lemmainfgrpcoh}(a) to $F$,
\vspace{-5pt}

\begin{center}\begin{tikzcd}
    \mathrm{H}^1(k,\widehat{F})\arrow[r]\arrow[d] & \mathrm{H}^2(k,\widehat{V})\arrow[r]\arrow[d, "\beta_v"] & \mathrm{H}^2(k,\widehat{W})\arrow[r]\arrow[d, "\gamma_v"] & 0\arrow[d]\\
    \mathrm{H}^1(k_v,\widehat{F})\arrow[r, "p"] & \mathrm{H}^2(k_v,\widehat{V})\arrow[r] & \mathrm{H}^2(k_v,\widehat{W})\arrow[r] &
    \mathrm{H}^2(k_v,\widehat{F})
\end{tikzcd}\end{center}\noindent
and then note that $F$ is of multiplicative type and $V$ is unipotent. We use Lemma \ref{lemmaconnhomcont} for $W$ to get that the map $p : \mathrm{H}^1(k_v,\widehat{F})\rightarrow \mathrm{H}^2(k_v,\widehat{V})$ is continuous. Since $\mathrm{H}^2(k_v,\widehat{V})$ is discrete, the image $p(Z)$ of any subset $Z\subseteq \mathrm{H}^1(k_v,\widehat{F})$ agrees with the image $p(\overline{Z})$ of its closure. It follows from this, because $\beta_v$ is injective, that we may replace $\mathrm{H}^1(k,\widehat{F})$ in the diagram by the closure of its image in $\mathrm{H}^1(k_v,\widehat{F})$, without changing the exactness. The first column is now surjective by Lemma \ref{lemmainfgrpcoh}(b), so injectivity of $\gamma_v$ follows from the 4-lemma.

\textit{Step 3:} 
It remains only to prove that $\beta_v : \mathrm{H}^2(k,\widehat{V})\rightarrow \mathrm{H}^2(k_v,\widehat{V})$ is injective for a connected unipotent group $V$ with no non-finite wound quotients. If $V$ is finite, then it's infinitesimal and injectivity holds since $\mathrm{H}^2(k,\widehat{V}) = 0$ by Lemma \ref{lemmainfgrpcoh}(a). Otherwise, we proceed by induction on the dimension of $V$. The group $V$ is not wound, so by definition, there is a monomorphism $\mathbf G_{\mathrm a}\hookrightarrow V$. The group $U\coloneqq V/\mathbf G_{\mathrm a}$ is connected and has no non-finite wound quotients, so the map $\alpha_v : \mathrm{H}^2(k,\widehat{U})\rightarrow \mathrm{H}^2(k_v,\widehat{U})$ is, by the induction hypothesis, injective in the diagram
\begin{center}\begin{tikzcd}
    \mathrm{H}^1(k,\widehat{\mathbf G_{\mathrm a}})\arrow[r]\arrow[d] & \mathrm{H}^2(k,\widehat{U})\arrow[r]\arrow[d, "\alpha_v"] & \mathrm{H}^2(k,\widehat{V})\arrow[r]\arrow[d, "\beta_v"] & \mathrm{H}^2(k,\widehat{\mathbf G_{\mathrm a}})\arrow[d]\\
    \mathrm{H}^1(k_v,\widehat{\mathbf G_{\mathrm a}})\arrow[r, "p"] & \mathrm{H}^2(k_v,\widehat{U})\arrow[r] & \mathrm{H}^2(k_v,\widehat{V})\arrow[r] &
    \mathrm{H}^2(k_v,\widehat{\mathbf G_{\mathrm a}})
\end{tikzcd}\end{center}\noindent
with exact rows. The two groups on the left are equal to $0$ by Proposition 2.5.3(iii) in \cite{RosTD}. Finally, a canonical isomorphism $\mathrm{H}^2(K,\widehat{\mathbf G_\mathrm{a}})\cong\Omega^1_K$ is shown in \cite{RosTD}, Corollary 2.7.3, for any $K$ such that $[K : K^p] = p$. Now, the right-most vertical map is an injection by separability of $k_v/k$ and by \cite{Mat90}, Theorem 26.6(3).
\end{proof}

\begin{cor}
Let $G$ be as in the preceding theorem. If $G$ is connected and admits no torus as a subgroup, then weak approximation for $\Omega$ and strong approximation for any proper subset $S\neq\Omega$ of not-necessarily-finite places hold for (any torsor of) $G$.
\end{cor}
\begin{proof}
If $\mathrm{char}\,k = 0$, then every connected multiplicative group is a torus. Thus $G$ is unipotent by Lemma \ref{lemmatorunip}, hence split unipotent, and the statement reduces to the classical strong approximation in $\mathbf G_{\mathrm{a}}(\mathbf A) = \mathbf A$.
If $\mathrm{char}\,k > 0$, then Steps 2 and 3 in the proof of the preceding theorem show that $\mathrm{H}^2(k,\widehat{G})\rightarrow \mathrm{H}^2(k_v,\widehat{G})$ is injective for every $v$.  Then $\Sh_S^2(\widehat{G}) = 0$ for $S = \Omega{\setminus}\{v\}$ and strong approximation holds with respect to $S$ by Proposition \ref{prop2}.
\end{proof}

\begin{cor}
Let $G$ be a commutative affine algebraic group over a global field $k$. Then the sequence \eqref{eq4} is an exact sequence of Abelian groups.
\end{cor}
\begin{proof}
By Remark \ref{reminvlim}, sequence \eqref{eq4} is the inverse limit of the following exact sequences
$$0\longrightarrow \overline{G(k)}^{_S}\longrightarrow G(k_S)\xrightarrow{\;\;f'_S\;\;} \Sh_S^2(\widehat{G})^*\xrightarrow{\;\;g_S\;\;} \Sh^2(\widehat{G})^*\longrightarrow 0$$
taken over finite $S\subseteq\Omega_f$.
Consider the sequences $0\rightarrow \overline{G(k)}^{_S}\rightarrow G(k_S)\rightarrow G(k_S)/\overline{G(k)}^{_S}\rightarrow 0$~for finite $S\subseteq\Omega_f$. It suffices to show that their inverse limit remains exact, since this then implies:
$$G(k_{\Omega_f})/\overline{G(k)} = \varprojlim\nolimits_{\textrm{finite }S\subseteq\Omega_f}\big(G(k_S)/\overline{G(k)}^{_S}\big) = \mathrm{ker}(\Sh_f^2(\widehat{G})^*\!\rightarrow\!  \Sh^2(\widehat{G})^*)$$
To prove the right-exactness of an inverse limit, we check the Mittag-Leffler property: In fact, we will prove the stronger property that $\overline{G(k)}^{_T}\rightarrow \overline{G(k)}^{_S}$ is a surjection for finite sets $S\subseteq T\subseteq\Omega_f$.
It suffices to prove this for $T = S\cup\{v\}$, so consider the following diagram with exact rows:
\begin{center}\begin{tikzcd}
    0\arrow[r] & G(k_v)\arrow[r]\arrow[d] & G(k_T)\arrow[r]\arrow[d] & G(k_S)\arrow[r]\arrow[d] & 0\\
    0\arrow[r] & \mathrm{im}\left(\mathrm{H}^2(k_v,\widehat{G})^*\rightarrow\Sh^2_T(k,\widehat{G})^*\right)\arrow[r] & \Sh^2_T(k,\widehat{G})^*\arrow[r] & \Sh^2_S(k,\widehat{G})^*\arrow[r] &
    0
\end{tikzcd}\end{center}\noindent
Now $\overline{G(k)}^{_T}\rightarrow \overline{G(k)}^{_S}$ appears as the map between kernels of the right two vertical morphisms. By local duality, $G(k_v)$ is dense in $G(k_v)\pro = \mathrm{H}^2(k_v,\widehat{G})^*$. The snake lemma thus implies that, to end the proof, it's sufficient to prove that the image of $G(k_v)$ in $\Sh^2_T(k,\widehat{G})^*$ is compact. This image is the quotient of $G(k_v)$ by the intersection $G(k_v)\cap\overline{G(k)}^{_T}$ in $G(k_T)$.

For this, we first note that, by Lemma 3.1.1 in \cite{Con12}, we may replace $G$ by a smooth subgroup $G'$ such that $G'(K) = G(K)$ for $K\in\{k,k_v\}$, where the equality preserves topology by definition of topology on $K$-points. Now, take a short exact sequence of smooth algebraic groups $0\rightarrow H\rightarrow G\rightarrow U\rightarrow 0$ such that $U$ is a wound unipotent group and $H$ has no wound unipotent quotient (which is possible by \cite{Oes84}, V.5, by taking $H$ to be an extension of a split unipotent group by the smooth multiplicative part of $G$). The commutative diagram, in which we embed $G(k_v)$ into $G(k_T)$,
\vspace{-10pt}

\begin{center}\begin{tikzcd}
    0\arrow[r] & H(k_v)\cap\overline{H(k)}^{_T} \arrow[r]\arrow[d] & G(k_v)\cap\overline{G(k)}^{_T}\arrow[d]\\
    0\arrow[r] & H(k_v)\arrow[r] & G(k_v)\arrow[r] & U(k_v)
\end{tikzcd}\end{center}
\vspace{-5pt}

\noindent
gives, after dividing by $H(k_v)\cap\overline{H(k)}^{_T}$, a short exact sequence of groups
$$0\longrightarrow \frac{H(k_v)}{H(k_v)\cap\overline{H(k)}^{_T}}\longrightarrow \frac{G(k_v)}{H(k_v)\cap\overline{H(k)}^{_T}}\longrightarrow U(k_v)$$
\vspace{-5pt}

\noindent
with last map open by smoothness of $H$ (and by \cite{CesTC}, Proposition 4.3, or alternatively by the discussion preceding Theorem 4.5 in \cite{ConAd}). The group on the left is finite, since it's a subgroup of $\Sh^2_T(k,\widehat{H})^*$, which is finite by Theorem \ref{thmsha2fin}. On the other hand, $U(k_v)$ is compact by Theorem VI.2.1 in \cite{Oes84}. We deduce that the origin of the surjective continuous map
\vspace{-7pt}

\begin{center}\begin{tikzcd}
    \dfrac{G(k_v)}{H(k_v)\cap\overline{H(k)}^{_T}}\arrow[r, two heads] & \dfrac{G(k_v)}{G(k_v)\cap\overline{G(k)}^{_T}}
\end{tikzcd}\end{center}
\vspace{-7pt}

\noindent
is compact, hence its target must be as well.
\end{proof}

\begin{cor}
Let $G$ be a commutative affine algebraic group over $k$. The topological group $G(k_{\Omega_f})/\overline{G(k)}$ is compact. Equivalently, its inclusion into $\Sh_f^2(\widehat{G})^*$ is a topological embedding.
\end{cor}
\begin{proof}
The Mittag-Leffler property shown in the preceding proof gives a group isomorphism $G(k_{\Omega_f})/\overline{G(k)}\xrightarrow{\,\sim\,} \varprojlim_S\big(G(k_S)/\overline{G(k)}^{_S}\big)$, where the limit (taken over finite $S\subseteq\Omega_f$) on the right is compact Hausdorff. The isomorphism is continuous and, to show that it's a homeomorphism, we just need to prove that it is also open.

A basis of the topology on $G(k_{\Omega_f})$ is given by open sets $U = U_T\times G(k_{\Omega_f\setminus T})$, for some finite set $T$ and open $U_T\subseteq G(k_T)$. A subset of $G(k_{\Omega_f})/\overline{G(k)}$ is open if and only if its preimage in $G(k_{\Omega_f})$ is. Thus it suffices to prove that the image of the set $\widetilde{U}\coloneqq (U_T\times G(k_{\Omega_f\setminus T}))+\overline{G(k)}\subseteq G(k_{\Omega_f})$ is open in $\varprojlim_{S}\big(G(k_S)/\overline{G(k)}^{_S}\big)$. 

In the proof of the preceding corollary, it was shown that $\overline{G(k)}^{_S}\!\!\rightarrow\overline{G(k)}^{_T}$ is a surjection~for~any finite $S\supseteq T$ with $S\subseteq\Omega_f$. Remark \ref{reminvlim} now implies that the map $\overline{G(k)} = \varprojlim_{S}\overline{G(k)}^{_S}\!\longrightarrow\overline{G(k)}^{_T}$ is also a surjection. This proves that $\widetilde{U} = (U_T+\overline{G(k)}^{_T})\times G(k_{\Omega_f\setminus T})$. Since both the set $\widetilde{U}$~and~its complement are $\overline{G(k)}$-invariant, they have disjoint projections to $G(k_T)/\overline{G(k)}^{_T}$.

The image of $\widetilde{U}$ in the limit $\varprojlim_{S}\big(G(k_S)/\overline{G(k)}^{_S}\big)$ is thus exactly the preimage of its open~image $\big(U_T+\overline{G(k)}^{_T}\big)/\overline{G(k)}^{_T}$ in the quotient $G(k_T)/\overline{G(k)}^{_T}$. Hence, it is open.
\end{proof}

\begin{cor}
Let $G$ be as in Theorem \ref{thmsha2fin}. If $G$ is connected, then $G(k_{\Omega})/\overline{G(k)}$~is~finite.
\end{cor}
\begin{proof}
Clear when $\mathrm{char}\,k > 0$, by $k_{\Omega} = k_{\Omega_f}$. Otherwise, it's Corollary 3.5(i) in \cite{San81}.
\end{proof}
\section{Explicit counterexamples}

The statement of Theorem \ref{thmsha2fin} from the previous section, contains two major assumptions:
\begin{itemize}
    \item finiteness of $\Sh_S^2(\widehat{G})$ is conditional on the property that: for any quotient group of $G$, if it is wound unipotent, then it must be finite
    \item finiteness of $\Sh_f^2(\widehat{G})$ is conditional on both the above property and, additionally, on the connectedness of $G$
\end{itemize}
The first assumption cannot be dropped, as shown in Remark \ref{remoes} by a general example that follows immediately from the work of Oesterl\'e and the exact sequence from \eqref{eq2}. The second assumption is classically known to be necessary, as is shown by any constant group $\mathbf Z/n, n\geq 2$. Indeed, if $G = \mathbf Z/n$, then $G(k)$ is finite and $\prod_{v\in\Omega_f} G(k_v)$ infinite, so that the exact sequence \eqref{eq3} immediately implies the result.
Note that the first assumption is relevant only in positive characteristic, while the second assumption is important regardless of characteristic and for more obvious reasons. 

Consider the global field $k = \mathbf F_p(t)$ of characteristic $p\geq 3$. Let $[t]$ denote the place~of~the prime $(t)$ of $k$. In this section we study the cases when $G = \mathbf Z/p$ (Example \ref{exexpl1}) and when~$G$~is~the group $\{tx^p = y^p-y\}$ from Remark \ref{remoes} (Example \ref{exexpl2}), to explicitly provide infinite families of distinct elements in $\Sh_\omega^2(\widehat{G}) = \Sh_f^2(\widehat{G})$ and $\Sh_{[t]}^2(\widehat{G})$, respectively, without using duality theory.

For this, we will use the canonical $K$-module isomorphism $\mathrm{H}^2(K, \widehat{\mathbf G_{\mathrm{a}}})\cong\Omega_K^1$ from \cite{RosTD}, Proposition 2.7.6, over an arbitrary field $K$ such that $[K : K^p] = p$, where $\Omega_K^1$ is the module of differentials over $K$. It is clear that, for the non-finite example above, we need to understand the endomorphism of $\Omega_K^1$ (induced through this identification) by the Frobenius endomorphism $\mathrm{Fr}_{\mathbf G_{\mathrm a}} : \mathbf G_{\mathrm a}\rightarrow\mathbf G_{\mathrm a}^{(p)} = \mathbf G_{\mathrm a}$ which acts on points as $x\mapsto x^p$. This is done in the proposition below. Before proving it, we must explain in more detail this $K$-module isomorphism:

Explicitly, the following composition of maps is shown by Rosengarten to be an isomorphism,
$$\Omega^1_K\rightarrow \Omega^1_{\mathbf G_{\mathrm a}}\twoheadrightarrow \frac{\Omega^1_{\mathbf G_{\mathrm a}}/B^1_{\mathbf G_{\mathrm a}}}{\mathrm{im}(C^{-1}-i)}\xrightarrow{\;\sim\;} \mathrm{H}^1(\mathbf G_{\mathrm a},\mathbf G_{\mathrm m}/\mathbf G_{\mathrm m}^p)\rightarrow \mathrm{Br}(\mathbf G_{\mathrm a})
\hookleftarrow
\mathrm{Ext}_K^2(\mathbf G_{\mathrm a},\mathbf G_{\mathrm m})\xleftarrow{\;\sim\;} \mathrm{H}^2(K,\widehat{\mathbf G_{\mathrm a}})$$
where the first map is $\omega\mapsto X\omega$ (for $\mathbf G_{\mathrm a, K} = \mathrm{Spec}(K[X])$), the second is the canonical surjection, the fourth is a connecting homomorphism coming from the $p$-th power map on $\mathbf G_{\mathrm m}$, the fifth is induced by a Yoneda inclusion and the last is an edge map in the Leray spectral sequence. The third map is induced by the connecting homomorphism coming from the exact sequence (in which we identify a module over a ring with its associated sheaf over the affine scheme)
$$0\longrightarrow \mathbf G_{\mathrm m}/\mathbf G_{\mathrm m}^p\xrightarrow{\;f\,\mapsto\,\mathrm{d}f/f\;} \Omega^1_{\mathbf G_{\mathrm a}}\xrightarrow{\;C^{-1}-i\;} \Omega^1_{\mathbf G_{\mathrm a}}/B^1_{\mathbf G_{\mathrm a}} \rightarrow 0$$
where $B^1_{\mathbf G_{\mathrm a}}$ is the image of $\mathrm{d} : \mathbf G_{\mathrm a}\rightarrow\Omega^1_{\mathbf G_{\mathrm a}}$, the map $i$ is the quotient map $i : \Omega^1_{\mathbf G_{\mathrm a}}\rightarrow \Omega^1_{\mathbf G_{\mathrm a}}/B^1_{\mathbf G_{\mathrm a}}$~and $C^{-1}$ is the ``formal inverse to the Cartier operator'', defined by $C^{-1}(f\,\mathrm{d}g) = [f^p g^{p-1}\,\mathrm{d}g]$.

\begin{prop}\label{propcart}
The map $C^{-1} : \Omega^1_K\rightarrow\Omega^1_K/B^1_K$ is a group isomorphism, and the composition $C\coloneqq (C^{-1})^{-1}\circ i \, :\, \Omega^1_K\rightarrow\Omega^1_K$ is the Cartier operator given by:
\vspace{-10pt}

$$C(f\,\mathrm{d}g) = f_{p-1}\,\mathrm{d}g\, ,\;\textrm{ where }\; f = \sum_{j=0}^{p-1} f_j^p g^j\;\textrm{ for } f_j\in K\textrm{, unique if }\mathrm{d}g\neq 0$$
\vspace{-3pt}

\noindent
There is an exact sequence $0\longrightarrow K^p\longrightarrow K\xrightarrow{\;\;d\;\;} \Omega^1_K\xrightarrow{\;\;C\;\;} \Omega^1_K\longrightarrow 0$. Furthermore, the diagram
\vspace{-12pt}

\begin{center}\begin{tikzcd}[column sep = large]
    \Omega^1_K\arrow[r, "\sim"]\arrow[d, "C"] & \mathrm{H}^2(K,\widehat{\mathbf G_{\mathrm a}})\arrow[d, "\mathrm{Fr}_{\mathbf G_{\mathrm a}}^*"]\\
    \Omega^1_K\arrow[r, "\sim"] & \mathrm{H}^2(K,\widehat{\mathbf G_{\mathrm a}})
\end{tikzcd}\end{center}\noindent
commutes, where the right column is induced by $\mathrm{Fr}_{\mathbf G_{\mathrm a}} : \mathbf G_{\mathrm a}\rightarrow\mathbf G_{\mathrm a}$.
\end{prop}
\begin{proof}
The first claim is true quite generally (for $\Omega^1_{X/S}$ where $X\rightarrow S$ is smooth), but it can easily be checked by hand in our case, using that $[K:K^p] = p$. In particular, surjectivity of $C^{-1}$ follows so: for any nonzero $[f\,\mathrm{d}g]\in\Omega^1_K/B^1_K$, we immediately have $g\notin K^p$ and $f = \sum_{j=0}^{p-1} f_j^p g^j$ uniquely. Then $C^{-1}(f_{p-1}\,\mathrm{d}g)$ is represented by
\vspace{-10pt}

$$f_{p-1}^pg^{p-1}\,\mathrm{d}g = f\,\mathrm{d}g - \sum_{j=0}^{p-2} f_j^p g^j\,\mathrm{d}g = f\,\mathrm{d}g - \mathrm{d}\!\left(\sum_{j=0}^{p-2} (j+1)^{-1}f_j^p g^{j+1}\right)$$
\vspace{-5pt}

\noindent
which proves the form of $C$. Lemmas 2.6.1 and 2.7.4 in \cite{RosTD} show exactness of the stated exact sequence, provided that the diagram is commutative. For this last fact, note first that all intermediate groups in the construction of the isomorphism $\Omega^1_K\xrightarrow{\sim}\mathrm{H}^2(K,\widehat{\mathbf G_{\mathrm a}})$ are functorial in $\mathbf G_a$, so it suffices to show commutativity of the diagram:
\vspace{-7pt}

\begin{center}\begin{tikzcd}[column sep = large]
    \Omega^1_K\arrow[r, "\omega\,\mapsto {[X\omega]}"]\arrow[d, "C"] & \dfrac{\Omega^1_{\mathbf G_{\mathrm a}}/B^1_{\mathbf G_{\mathrm a}}}{\mathrm{im}(C^{-1}-i)}\arrow[d, "\mathrm{Fr}_{\mathbf G_{\mathrm a}}^*"]\\
    \Omega^1_K\arrow[r, "\omega\,\mapsto {[X\omega]}"] & \dfrac{\Omega^1_{\mathbf G_{\mathrm a}}/B^1_{\mathbf G_{\mathrm a}}}{\mathrm{im}(C^{-1}-i)}
\end{tikzcd}\end{center}
\vspace{-7pt}

\noindent
Here, $\mathrm{Fr}_{\mathbf G_{\mathrm a}}^*([X\cdot f\,\mathrm{d}g]) = [X^p\cdot f\,\mathrm{d}g]$, since the relative Frobenius morphism $\mathrm{Fr}_{\mathbf G_{\mathrm a}}$ is the dual of the $K$-algebra map $(X\mapsto X^p) : K[X]\rightarrow K[X]$. However, the difference $[X\cdot C(f\,\mathrm{d}g) - X^p\cdot f\,\mathrm{d}g]$ is exactly $(C^{-1}-i)(-X\cdot C(f\,\mathrm{d}g))$ in $\Omega^1_{\mathbf G_{\mathrm a}}/B^1_{\mathbf G_{\mathrm a}}$, hence it vanishes in the quotient.
\end{proof}

\begin{exmp}
Taking duals in the short exact sequence $0\longrightarrow \alpha_p\longrightarrow \mathbf G_{\mathrm a}\xrightarrow{\mathrm{Fr}} \mathbf G_{\mathrm a}\longrightarrow 0$ gives the top row of the following commutative diagram:
\vspace{-3pt}

\begin{center}\begin{tikzcd}
    0\arrow[r] & \mathrm{H}^1(K,\widehat{\alpha_p})\arrow[r, "\delta"]\arrow[d, "\rotatebox{90}{$\sim$}", pos=0.4] & \mathrm{H}^2(K,\widehat{\mathbf G_{\mathrm a}})\arrow[r, "\mathrm{Fr}^*"]\arrow[d, "\rotatebox{90}{$\sim$}", pos=0.4] & \mathrm{H}^2(K,\widehat{\mathbf G_{\mathrm a}})\arrow[r]\arrow[d, "\rotatebox{90}{$\sim$}", pos=0.4] & 0\\
    0\arrow[r] & K/K^p\arrow[r, "?"] &
    \Omega^1_{K}\arrow[r, "C"] & \Omega^1_{K}\arrow[r] & 0
\end{tikzcd}\end{center}
\vspace{-3pt}

\noindent
The left-most vertical map is an isomorphism which can easily be shown using Čech cohomology. If the right two vertical maps are Rosengarten's functorial isomorphism, then the exact sequence proven in the proposition above is consistent with the bottom-row map ``$?$'' being exactly $\mathrm{d}$.
\end{exmp}


\begin{exmp}\label{exexpllem}
Let $u$ be the uniformizer in a completion $k_v$ of $k = \mathbf F_p(t)$. For a fixed $b\in k_v$, consider an equation $a\,\mathrm{d}u - C(a\,\mathrm{d}u) = b\,\mathrm{d}u$ with unknown $a\in k_v$. We write $a = \sum_j a_ju^j$ and $b = \sum_j b_ju^j$ for some $a_j,b_j\in\mathbf F_q$. This equation is equivalent to a system of equations
\vspace{-11pt}

$$a_{j-1}^p - a_{pj-1} = b_{j-1}^p$$
\vspace{-12pt}

\noindent
which splits into three systems, depending on whether $pj-1$ is greater, equal or less than $j-1$:
\begin{itemize}
    \item Over $j > 0$, the system has infinitely many solutions, indexed by all the different choices of $a_{j-1}\in\mathbf F_q$ for $p\nmid j > 0$, each of which is possible.
    \item For $j = 0$, the equation $a_{-1} - a_{-1}^{1/p} = b_{-1}$ may or may not have a solution, depending on $b_{-1}$. If it does, it has exactly $p$ solutions.
    \item Over $j < 0$, the system has exactly one solution. Indeed, there always exists $N\geq 0$ such that $b_{j-1} = 0$ for $j < -N$. Any potential solution must satisfy
\vspace{-8pt}

    $$a_{j-1} = a_{pj-1}^{1/p} = a_{p^2j-1}^{1/p^2} = \ldots = 0$$
\vspace{-9pt}

\noindent
    for $j < -N$. Setting $a_{j-1} = 0$ for $j < -N$, we are forced into a unique solution of the system by iterating $a_{j-1} = b_{j-1} + a_{pj-1}^{1/p}$ for $j\geq -N$.
\end{itemize}
In particular, the existence of a solution of the starting equation depends only on the second case, that of $b_{-1}$. This calculation will be important in the two main examples of this section: 
\end{exmp}

\begin{exmp}\label{exexpl1}
The algebraic group $\mathbf Z/p$ can be seen (most easily as a scheme-theoretic fiber) as the kernel of $\mathbf G_{\mathrm a}\xrightarrow{\mathrm{id}-\mathrm{Fr}} \mathbf G_{\mathrm a}$. Hence, if $v$ is a place of $k = \mathbf F_p(t)$, we get the commutative~diagram
\begin{center}\begin{tikzcd}
    \Omega^1_k\arrow[r, "q"]\arrow[d, hook] & \Omega^1_k\arrow[r]\arrow[d, hook] & \mathrm{H}^2(k,\widehat{G})\arrow[r]\arrow[d, "\beta_v"] & 0\arrow[d]\\
    \Omega^1_{k_v}\arrow[r, "q_v"] &
    \Omega^1_{k_v}\arrow[r] & \mathrm{H}^2(k_v,\widehat{G})\arrow[r] & 0
\end{tikzcd}\end{center}\noindent
in which $q$ and $q_v$ are given by $\mathrm{id}-C$. The left two vertical maps are inclusions by \cite{Mat90}, Theorem 26.6(3) (the transcendental extension $k_v/k$ is separable). 
Now, $\Sh_S^2(\widehat{G}) = \bigcap_{v\notin S}\ker\beta_v$ and $\Sh_\omega^2(\widehat{G}) = \bigcup_{S\textrm{ finite}}\Sh_S^2(\widehat{G})$. Recall that $p\geq 3$. We define elements $x_N\in\Omega^1_k$ for $N\geq 0$ by
$$x_N\coloneqq\frac{t^{(p-1)^N-1}\mathrm{d}t}{1-t^{(p-1)^N}} = \sum_{n=1}^\infty t^{n\cdot (p-1)^N-1}\,\mathrm{d}t$$
If we can show that $x_N\in\mathrm{im}(q_v)$ for almost all $v$, then it represents an element $[x_N]\in\Sh_\omega^2(\widehat{G})$. If we then show that $x_N-x_K\notin\mathrm{im}(q)$ holds for all pairs $K > N$, we will have shown that the elements $[x_N]\in\Sh_\omega^2(\widehat{G})$ are distinct and that $\Sh_\omega^2(\widehat{G})$ is infinite. The proofs of these two statement are as follows:
\begin{enumerate}[\hspace{-0.07cm} 1)]

\item Let $u\in k_v$ be a uniformizer and write $x_N = \sum x_{N,j}u^j$ for $x_{N,j}\in\mathbf F_q$. For almost all $v$, we have $x_N\in\mathcal O_v$. Then $x_{N,-1} = 0-0^{1/p}$, hence Example \ref{exexpllem} shows that $x_N$ is in the image of $q_v = \mathrm{id}-C$. This is what we wanted to prove.

As a remark, note that another way to write down the conclusion of Example \ref{exexpllem} is to say that the class $[b\,\mathrm{d}u]\in\mathrm{coker}(q_v)$, for $b = \sum b_ju^j\in k_v$, depends only on $b_{-1}\in\mathbf F_q$.
More precisely, if $f : \mathbf F_q\rightarrow\mathbf F_q$ denotes the homomorphism $x\mapsto x-x^{1/p}$, then it immediately follows that $\texttt{\#}\mathrm{coker}(q_v) = \texttt{\#}\mathrm{coker}(f) = \texttt{\#}\mathrm{ker}(f) = p$. This verifies that $\mathrm{H}^2(k_v,\widehat{\mathbf Z/p})\simeq\mathbf Z/p$, which is consistent with local duality.

\item Now, let $v = [t]$. Note that the image of $k\hookrightarrow k_v$ consists exactly of those $\sum_n a_n t^n$ for which the sequence $a_n$ is eventually periodic: Indeed, every eventually periodic sequence gives a series clearly equal to a rational function. For the converse, take a rational function $P(t)/Q(t)$ with $P,Q$ relatively prime. We can assume $t\nmid Q$. Because $t$ and $Q$ are relatively prime in the PID $\mathbf F_p[t]$, the image of $t$ is a unit in the finite quotient ring $\mathbf F_p[t]/(Q)$. Hence $t^m = 1+Q(t)R(t)$ for some $m > 0$ and $R\in\mathbf F_p[t]$, so $P(t)/Q(t) = -P(t)R(t)\sum_n t^{mn}$, which has eventually periodic coefficients.

Let $K > N\geq 0$. One way to show that $x_N-x_K\notin\mathrm{im}(q)$ is to show that, given some $a = \sum a_n t^n\in k_v$ such that $a\,\mathrm{d}t - C(a\,\mathrm{d}t) = x_N-x_K$, the sequence of $a_n$ cannot eventually be periodic. For the sake of contradiction, suppose that it is: Then we have
$$\sum_{n=1}^\infty t^{n\cdot (p-1)^N-1}\,\mathrm{d}t - \sum_{n=1}^\infty t^{n\cdot (p-1)^K-1}\,\mathrm{d}t = \sum_n a_nt^n\,\mathrm{d}t-C\!\left(\sum_n a_nt^n\,\mathrm{d}t\!\right) \!= \sum_n (a_{n-1}-a_{np-1})t^{n-1}\,\mathrm{d}t$$
for a sequence $a_n$ which is eventually periodic with period $P$, where $a_n^p = a_n$ since $a_n\in\mathbf F_p$. Suppose that $P\mid M\cdot (p-1)^r$ for some $M,r$ with $\mathrm{gcd}(M,p-1) = 1$. Looking at the coefficients next to $t^n$ for $n = M(p-1)^Np^j$, the above identity shows $a_{M(p-1)^Np^j-1} - a_{M(p-1)^Np^{j+1}-1} = 1$ for $j\geq 0$ (precisely because $(p-1)^K$ does not divide $M(p-1)^Np^j$). Summing the resulting expressions for $j\in\{0,\ldots,s-1\}$ gives $a_{M(p-1)^N-1} - a_{M(p-1)^Np^s-1} = s$ for any $s\geq 0$.

Now, $(M(p-1)^Np^s-1) - (M(p-1)^N-1) = M(p-1)^N(p^s-1)$, which is divisible by $P$ if $(p-1)^r\mid p^s-1$. This is true if we let $s = \varphi((p-1)^r)$, where $\varphi$ denotes Euler's totient function. Finally, we use that $a_n$ is eventually periodic: For a large enough choice of $M$, we get $a_{M(p-1)^N-1} = a_{M(p-1)^Np^s-1}$ in $\mathbf F_p$, i.e. $p\mid s$. However, all prime factors of the totient function $\varphi((p-1)^r)$ are smaller than $p$, a contradiction.
\end{enumerate}
\end{exmp}

This completes our analysis of one example complementing Theorem \ref{thmsha2fin}, as stated at the beginning of the section. In this example, the difficulty in constructing an infinite sequence of distinct elements in $\mathrm{coker}(q)$ lies in the fact that they cannot all be distinct in any $\mathrm{coker}(q_v)$ (which we've seen to be a finite group isomorphic to $\mathbf Z/p$). However, there is no such difficulty in the following example:

\begin{exmp}\label{exexpl2}
The algebraic group $G = \{tx^p = y^p-y\}$ over $k = \mathbf F_p(t)$ is the kernel of the map $\mathbf G_{\mathrm a}^2\xrightarrow{(x,y)\;\mapsto\; tx^p-y^p+y} \mathbf G_{\mathrm a}$. If $v$ is a place of $k = \mathbf F_p(t)$, we get a commutative diagram as before:
\vspace{-5pt}

\begin{center}\begin{tikzcd}
    \Omega^1_k\arrow[r, "q"]\arrow[d, hook] & (\Omega^1_k)^2\arrow[r]\arrow[d, hook] & \mathrm{H}^2(k,\widehat{G})\arrow[r]\arrow[d, "\beta_v"] & 0\arrow[d]\\
    \Omega^1_{k_v}\arrow[r, "q_v"] &
    (\Omega^1_{k_v})^2\arrow[r] & \mathrm{H}^2(k_v,\widehat{G})\arrow[r] & 0
\end{tikzcd}\end{center}
\vspace{-3pt}

\noindent
Here, $q$ and $q_v$ are given by $\eta\mapsto (C(t\eta),\, \eta-C(\eta))$. We define elements $x_N = (t^{-N}\,\mathrm{d}t,\, 0)\in(\Omega^1_k)^2$ for $N > 0$.
This time, we first show that $x_N-x_K\notin\mathrm{im}(q_{[t]})$, where $[t]$ is the place of the prime $(t)$; it follows in particular that $x_N-x_K\notin\mathrm{im}(q)$. Then we show that $x_N\in\mathrm{im}(q_v)$ for all $v\neq [t]$. In this way we will have constructed infinitely many distinct elements $[x_N]\in\Sh_{[t]}^2(\widehat{G})$. Recall that $p\geq 3$. The proofs of these two assertions follow, in order:
\begin{enumerate}[\hspace{-0.07cm} 1)]

\item Let $K > N > 0$ and suppose that $q_{[t]}(a\,\mathrm{d}t) = x_N-x_K$ for $a = \sum_j a_jt^j$. Then $a\,\mathrm{d}t = C(a\,\mathrm{d}t)$ and $C(at\,\mathrm{d}t) = (t^{-N}-t^{-K})\,\mathrm{d}t$.
However, Example \ref{exexpllem} immediately shows us that $a_{j-1} = 0$ when $j < 0$, uniquely determined by the first equality. Thus $C(at\,\mathrm{d}t) = \sum_j b_jt^j$ with $b_j = 0$ for all $j < 0$, a contradiction. 

\item Let $v\neq [t]$ be a place of $k$. We can choose a uniformizer $u\in k_v$. We want to show that $q_v(a\,\mathrm{d}u) = x_N$ for some $a\in k_v$. Since $[k_v : k_v^p] = p$ and $u\notin k_v^p$, every $a$ can be written in the form $a = \sum_{j=0}^{p-1} a_j^p u^j$ with $a_j\in k_v$.

First, suppose that $v = [1/t]$ and $u = 1/t$. The system can be rewritten as $a\,\mathrm{d}u = C(a\,\mathrm{d}u)$ and $C(a\,\mathrm{d}u/u) = -u^{N-2}\,\mathrm{d}u$. We will search for $a$ in the form $a = a_0^p+a_{p-1}^pu^{p-1}$. The second equation gives $a_0 = -u^{N-1}\in\mathcal O_v$. The first equation then equivalently becomes:
\vspace{-12pt}

$$(a_{p-1}u)^p - a_{p-1}u = -a_0^pu$$
\vspace{-12pt}

\noindent
An application of Hensel's lemma to the polynomial $X^p - X = u^{p(N-1)+1}$ with single root $0$ modulo $u$ gives us the desired element $a_{p-1}$.

Otherwise, suppose that $v\neq [1/t]$ and that $u\in\mathbf F_p[t]$ is some irreducible polynomial in $t$. In this case, crucially, we see that $v(t) = 0$. It follows that $t\in\mathcal O_v$ and that we can write $t = \sum_{j=0}^{p-1} t_j^p u^j$ for $t_j\in\mathcal O_v$. Next, write $x_N = (y_N\,\mathrm{d}u,\, 0)$, where $y_N = t^{-N}t'_u\in\mathcal O_v$. We will search for $a$ in the form $a = a_{p-2}^pu^{p-2}+a_{p-1}^pu^{p-1}$ with $a_{p-2},a_{p-1}\in\mathcal O_v$. However, unlike the previous case, the solutions to both equations in the system depend on each other.

For a fixed $a_{p-2}\in\mathcal O_v$, the equation $a\,\mathrm{d}u = C(a\,\mathrm{d}u)$ becomes
\vspace{-12pt}

$$(a_{p-1}u)^p - a_{p-1}u = -a_{p-2}^pu^{p-1}$$
\vspace{-12pt}

\noindent
which has, just as before, a unique solution $a_{p-1}u\in\mathcal O_v$ which is $0$ modulo $u$. Equivalently, this determines a unique $a_{p-1}\coloneqq f(a_{p-2})\in\mathcal O_v$.
Explicitly, this lift is given by the series
\vspace{-8pt}

$$f(a_{p-2}) = -u^{-1}\sum_{s=0}^\infty \left(-a_{p-2}^p u^{p-1}\right)^{p^s} = \sum_{s=0}^\infty a_{p-2}^{p^{s+1}} u^{(p-1)p^s-1}$$
\vspace{-6pt}

\noindent
whose $n$-th partial sum we denote by $f_n$. Then $f_n(a_{p-2})\longrightarrow f(a_{p-2})$.

It remains to prove the existence of $a_{p-2}\in\mathcal O_v$ such that setting $a_{p-1} = f(a_{p-2})$ gives us $C(at\,\mathrm{d}u) = y_N$. This condition is equivalent to $a_{p-2}t_1+f(a_{p-2})t_0 = y_N$, by looking at the $u^{p-1}$-coefficient of $at$. By compactness of $\mathcal O_v$, it suffices to find solutions of all the analogous equations with $f$ replaced by $f_n$, for all $n$. Note that $\partial f_n/\partial a_{p-2} = 0$, since all powers of $a_{p-2}$ appearing in the expansion are divisible by $p$. Also: $f_n(a_{p-2})\equiv 0\;(\mathrm{mod}\, u)$

Define $P_n(a_{p-2}) = a_{p-2}t_1+f_n(a_{p-2})t_0-y_N$. Then $\partial P_n/\partial a_{p-2} = t_1$ and:
\vspace{-11pt}

$$P_n(a_{p-2})\equiv a_{p-2}t_1-y_N\;(\mathrm{mod}\, u)$$
\vspace{-14pt}

\noindent
All that remains to apply Hensel's lemma is to prove that $t_1\notin\mathfrak m_v$. Suppose the contrary. Then $t = t_0^p + ru^2$ for some $r\in\mathcal O_v$. Writing $u = P(t) = P(t_0^p + ru^2) = c^p + du^2$ for $c,d\in\mathcal O_v$ (since $P\in\mathbf F_p[t]$), we get a contradiction with $u\not\equiv c^p\;(\mathrm{mod}\, u^2)$. This ends the proof.
\end{enumerate}
\end{exmp}

\begin{rem}
Conceptually, the choice to pick the elements $x_N = (t^{-N}\,\mathrm{d}t,\, 0)$ can be explained using the duality theorems:
The identification $\mathrm{H}^2(k_v, \widehat{\mathbf G_{\mathrm{a}}})\cong\Omega_{k_v}^1$ implies that there is a canonical pairing $k_v\times\Omega_{k_v}^1\rightarrow\mathbf Q/\mathbf Z$, induced by the cup product. A natural possibility is the pairing taking $(x, f\,\mathrm{d}u)$ to $q^{-1}\mathrm{res}(xf\,\mathrm{d}u)\in q^{-1}\mathbf Z/\mathbf Z\subseteq\mathbf Q/\mathbf Z$ via the residue map $\mathrm{res} : f\,\mathrm{d}u\mapsto f_{-1}\in\mathbf F_q$ (where $q$ is the cardinality of the residue field at $v$) which is easily seen to be independent of the choice of uniformizer $u$. It even makes the square in Proposition \ref{propcart} commute.

By Remark \ref{remoes}, the group $\Sh_{[t]}^2(\widehat{G})$ is infinite precisely because the map $G(\mathbf A_{[t]})\rightarrow\Sh_{[t]}^2(\widehat{G})^*$ has infinite image. We would thus like to construct infinitely many pairs $[(f\,\mathrm{d}t,\,g\,\mathrm{d}t)]\in\Sh_{[t]}^2(\widehat{G})$ with the maps $(x,y)\mapsto\mathrm{res}((xf+yg)\,\mathrm{d}t)$ distinct on $G(\mathbf A_{[t]}) = \{tx^p = y^p-y\mid x,y\in k_{[t]}\}$.

Finally, we also know by Remark \ref{remoes} that such a pair $(x,y)$ can be found above every $x\in\mathcal O_{[t]}$. Thus there is an abundance of elements $(x,y)\in\Sh_{[t]}^2(\widehat{G})$ concentrated in positive degrees, but our pairing can only distinguish residues in degree $-1$. We are hence naturally led to consider multiplying $x\in\mathcal O_v$ by elements $t^{-n}\,\mathrm{d}t$ concentrated in negative degrees with respect to $t$.
\end{rem}
\linespread{0.6}

\end{document}